\documentclass{birkart}
\usepackage{amsmath}
\usepackage{amssymb}
\usepackage{amsfonts}
\usepackage{graphicx}
\usepackage{texdraw}
\usepackage{graphpap}
\usepackage{wasysym}
\usepackage{pxfonts}
\usepackage[OT2,T1]{fontenc}
\usepackage[utf8]{inputenc}
\DeclareSymbolFont{cyrletters}{OT2}{wncyr}{m}{n}
\DeclareMathSymbol{\berd}{\beta}{cyrletters}{"42}
\DeclareMathSymbol{\Zhe}{\beta}{cyrletters}{"11}

\input txdtools

 \newtheorem{thm}{Theorem}[section]
 \newtheorem{cor}[thm]{Corollary}
 \newtheorem{lem}[thm]{Lemma}
 \newtheorem{prop}[thm]{Proposition}
 \theoremstyle{definition}
 \newtheorem{defn}[thm]{Definition}

 \theoremstyle{remark}
 
 \newtheorem{rem}[thm]{Remark}

\numberwithin{equation}{section}
\numberwithin{figure}{section}

\newcommand{\e}{\mathrm e}

\newcommand{\Rpar}{q}

\newcommand{\C}{{\mathbb C}}
\newcommand{\D}{{\mathbb D}}

\newcommand{\R}{{\mathbb R}}

\newcommand{\ordo}{\mathrm{o}}
\newcommand{\Ordo}{\mathrm{O}}

\newcommand{\re}{\operatorname{Re}}
\newcommand{\im}{\operatorname{Im}}

\newcommand{\Tope}{{\mathbf T}}

\newcommand{\Bop}{{\mathbf B}}

\newcommand{\Fspace}{{\delta\mathrm{Pol}}}

\newcommand{\diff}{{\mathrm d}}
\newcommand{\imag}{{\mathrm i}}

\newcommand{\pv}{\operatorname{\mathrm{pv}}}

\begin{document}
%
\title[The Polyanalytic Ginibre Ensembles]
{The Polyanalytic Ginibre Ensembles}

\author[Haimi]
{Antti Haimi}

\address{Haimi: Department of Mathematics\\
The Royal Institute of Technology\\
S -- 100 44 Stockholm\\
SWEDEN}

\email{anttih@math.kth.se}



\author[Hedenmalm]
{Håkan Hedenmalm}

\address{Hedenmalm: Department of Mathematics\\
The Royal Institute of Technology\\
S -- 10044 Stockholm\\
SWEDEN}

\email{haakanh@kth.se}

\thanks{The second author is supported by the G\"oran Gustafsson Foundation 
(KVA) and Vetenskapsr\aa{}det (VR)}




\begin{abstract} 
For integers $n,\Rpar=1,2,3,\ldots$, let $\mathrm{Pol}_{n,\Rpar}$ denote the 
$\C$-linear space of polynomials in $z$ and $\bar z$, of degree $\le n-1$ in 
$z$ and of degree $\le\Rpar-1$ in $\bar z$. We supply $\mathrm{Pol}_{n,\Rpar}$ 
with the inner product structure of
\[ 
L^2\left(\C,\e^{-m|z|^2} \diff A\right),\quad\text{where}\,\,\, 
\diff A(z)=\frac{\diff x \diff y}{\pi},\,\,\,z= x+ \imag y;
\]
the resulting Hilbert space is denoted by $\mathrm{Pol}_{m,n,\Rpar}$. Here,
it is assumed that $m$ is a positive real. We let $K_{m,n,\Rpar}$ denote the 
reproducing kernel of $\mathrm{Pol}_{m,n,\Rpar}$, and study the associated
determinantal process, in the limit as $m, n\to+\infty$ while $n=m+\Ordo(1)$; 
the number $\Rpar$, the degree of polyanalyticity, is kept fixed.
We call these processes polyanalytic Ginibre ensembles, because they 
generalize the Ginibre ensemble -- the eigenvalue process of random (normal) 
matrices with Gaussian weight. A possible interpretation is that we permit a
few higher Landau levels. We consider local blow-ups of the polyanalytic
Ginibre ensembles around points in the {\em spectral droplet}, which is here
the closed unit disk $\bar\D:=\{z\in\C:|z|\le1\}$. We obtain asymptotics for
the blow-up process, using a blow-up to characteristic distance $m^{-1/2}$;
the typical distance is the same both for interior and for boundary points
of $\bar\D$.  
This amounts to obtaining the asymptotical behavior of the generating kernel
$K_{m,n,q}$. Following \cite{ahm1}, the asymptotics of the $K_{m,n,q}$ is
rather conveniently expressed in terms of the Berezin measure (and density)
\[ 
\diff B^{\langle z\rangle}_{m,n,\Rpar}(w):=
\berd_{m,n,\Rpar}^{\langle z\rangle}(w)\diff A(w),\quad
\berd_{m,n,\Rpar}^{\langle z\rangle}(w)= 
\frac{|K_{m,n,\Rpar}(z,w)|^2 }{K_{m,n,\Rpar}(z,z)}\e^{-m|z|^2}. 
\]
For interior points $|z|<1$, we obtain that  
$
\diff B^{\langle z\rangle}_{m,n,\Rpar}(w)\to \diff\delta_z 
$
in the weak-star sense, where $\delta_z$ denotes the unit point mass at $z$. 
Moreover, if we blow up to the scale of $m^{-1/2}$ 
around $z$, we get convergence to a measure which is Gaussian for $\Rpar=1$, 
but exhibits more complicated Fresnel zone behavior for $\Rpar>1$. 
In contrast, for exterior points $|z|>1$, we have instead that 
$ 
\diff B^{\langle z\rangle}_{m,n,\Rpar}(w) \to \diff \omega(w,z, \mathbb{D}^e) 
$,
where $\diff \omega(w,z,\mathbb{D}^e)$ is the harmonic measure at $z$  
with respect to the exterior disk $\mathbb{D}^e:= \{w\in\C:\, |w|>1\}$. 
For boundary points, $|z|=1$, the Berezin measure 
$\diff B^{\langle z\rangle}_{m,n,\Rpar}$ 
converges to the unit point mass at $z$, like for interior points, but the 
blow-up to the scale $m^{-1/2}$ exhibits quite different behavior at boundary
points compared with interior points. The Fresnel-type pattern appears also 
for boundary points when $\Rpar>1$, but then it is not rotationally symmetric. 
\end{abstract}

\maketitle

\addtolength{\textheight}{2.2cm}







\section{Introduction}

\noindent{\bf Notation.} 
We will write use standard notation, such as $\partial X$ and $\mathrm{int}(X)$
for the boundary and the interior of a subset $X$ of the complex plane $\C$.
The complex conjugate of a complex number $z$ is usually written as 
$\bar z$. 
We write $\R$ for the real line, $\D:=\{z\in\C:|z|<1\}$ for the open 
unit disk, and $\D^e:=\{z\in\C:|z|>1\}$ for the open exterior (punctured) 
disk. The characteristic function of a set $E$ is written $1_E$. We write 
\[
\diff A(z)=\pi^{-1}\diff x \diff y,\quad\text{where}\,\,\,
z=x+\imag y\in\C,
\]
for the normalized area measure in $\C$, and use the standard Wirtinger 
derivatives
\[
\partial_z:=\tfrac12(\partial_x-\imag\partial_y),\,\,\,
\bar\partial_z:=\tfrac12(\partial_x+\imag\partial_y),
\quad\text{where}\,\,\, z=x+\imag y.
\] 
We also write $\Delta$ for the (quarter) Laplacian
\[
\Delta_z:=\partial_z\bar\partial_z=\tfrac14(\partial_x^2+\partial_y^2).
\]

\medskip
\noindent{\bf Determinantal projection processes.} 
Given a locally compact topological space $\mathrm{X}$ with a Radon measure 
$\mu$, a determinantal projection process (in the sequel just determinantal 
process) is a random configuration of $n$ points defined by the following 
probability measure on $\mathrm{X}^n$: 
\begin{equation}
\diff P(z_1,\ldots,z_n) = \frac{1}{n!} \det[K_n(z_i, z_j)]_{i,j=1}^n 
\diff \mu(z_1)\cdots\diff \mu(z_n). 
\label{eq-probdist}
\end{equation}
Here, $K_n$ is the integral kernel of a projection operator to an 
$n$-dimensional subspace of $L^2(\mathrm{X},\mu)$. It is customary to 
identify all the permutations of the points and think the process as a 
random measure $\sum_{j=1}^n \delta_{z_j}$ on $\mathrm{X}$.  

A general definition of a determinantal process was introduced by Macchi 
\cite{macchi}, who wanted to model fermions in quantum mechanics. Indeed, 
for any determinantal process, the probability density vanishes whenever any
two points in the $n$-tuple $(z_1,...,z_n)$ coincide (fermions are forbidden to
be in the same state). We interpret this as saying that the points in the 
$n$-tuple repel each other. Point processes of this kind appear in several 
contexts, e.g., in random matrix theory and combinatorics (for general 
surveys, s \cite{hkpv}, \cite{bor}; we should also mention the books
\cite{mehta}, \cite{deift}, \cite{deiftgioev}, \cite{AGZ}).  
\medskip

\noindent{\bf Eigenvalues of random normal matrix ensembles.}
Our main motivating example comes from the theory of random normal 
matrices.  This topic has in recent years been subject to rather active 
investigation by physicists as well as by mathematicians. 
For an introduction, see, e.g., \cite{zabro}. 
So, we shall use $\mathrm{X}=\mathbb{C}$ and 
$\diff\mu(z)= \e^{-mQ(z)}\diff A(z)$, for a positive weight function $Q$ 
satisfying some mild regularity and growth conditions; $m$ is a positive real 
parameter, and $\diff A(z)=\pi^{-1}\diff x \diff y$ is the normalized area 
measure. Let us write $L^2(\mathbb{C}, \e^{-mQ}):=L^2(\mathrm{X},\mu)$ in this
situation.
The determinantal projection process is associated with an $n$-dimensional
subspace of $L^2(\C, \e^{-mQ})$, and we will use the space $\mathrm{Pol}_{n}$ 
of all polynomials in $z$ of degree $\le n-1$; we write $\mathrm{Pol}_{m,n}$ 
to indicate that we have supplied $\mathrm{Pol}_{n}$ with the Hilbert space
structure of $L^2(\mathbb{C}, \e^{-mQ})$.
The density of the process is then given by the reproducing kernel  
$K_{m,n}$ of the space $\mathrm{Pol}_{m,n}$. So, we are talking about the
probability measure
\begin{equation}
\diff P(z_1,\ldots,z_n) = \frac{1}{n!} \det[K_{m,n}(z_i, z_j)]_{i,j=1}^n 
\e^{-m\{Q(z_1)+\cdots+Q(z_n)\}}\diff A(z_1)\cdots\diff A(z_n). 
\label{eq-probdist2}
\end{equation}
In terms of the {\em correlation kernel}
\begin{equation}
\Zhe_{m,n}(z,w):=K_{m,n}(z,w)\e^{-\frac12m\{Q(z)+Q(w)\}}, 
\label{eq-kernel-zhe}
\end{equation}
which is the an integral kernel of an orthogonal projection $L^2(\C)$, 
the expression \eqref{eq-probdist2} simplifies to
\begin{equation}
\diff P(z_1,\ldots,z_n) = \frac{1}{n!} \det\big[\Zhe_{m,n}(z_i, z_j)
\big]_{i,j=1}^n \diff A(z_1)\cdots\diff A(z_n). 
\label{eq-probdist2.5}
\end{equation}
The process described by \eqref{eq-probdist2} and \eqref{eq-probdist2.5}
represents the eigenvalues of a random normal matrix picked from the 
distribution 
\begin{equation}
\frac{1}{Z_{m,n}}\,\e^{-m\,\mathrm{tr}[Q(M)]} 
\diff\mathrm{vol}_{\mathrm{nm}(n)}(M),
\label{eq-ginQ}
\end{equation}
where 
$\diff\mathrm{vol}_{\mathrm{nm}(n)}(M)$ is the natural Riemannian volume 
form on the $n\times n$ normal matrices inherited from the metric of 
$\C^{n^2}$; $Z_{m,n}$ is the normalization constant needed to make the total 
mass equal to $1$. We are interested in the limiting behaviour of the 
process as $m,n\to+\infty$ while $n=m\tau+\Ordo(1)$ for some positive real 
number $\tau$. Without loss of generality, we will consider only $\tau=1$. 
\medskip

\noindent{\bf Local blow-up processes.}
Let $\mathcal{N}_+$ and $\mathcal{N}_{+,0}$ be the set of points defined by
\[
\mathcal{N}_+:=\big\{w\in\C:\,\Delta Q(w)>0\big\},\quad
\mathcal{N}_{+,0}:=\big\{w\in\C:\,\Delta Q(w)\ge0\big\}.
\]
In the arXiv preprint \cite{HedMak1}, which will appear later in the expanded 
form \cite{HedMak2}, the function $\widehat{Q}$ was defined as a certain 
envelope of $Q$, namely the largest subharmonic function in $\C$ which is 
$\le Q$ everywhere and has the growth bound  
\[
\widehat{Q}(z)=\log|z|^2+\Ordo(1),\quad\text{as}\,\,\,|z|\to+\infty.
\]
It is known that $\Delta \widehat{Q}=1_{\mathcal S}\Delta Q$ for some 
compact set $\mathcal S$ (see, e.g., \cite{HedMak2}). We assume that 
$\mathcal S$ is the minimal compact with this property, and call 
$\mathcal S$ a {\em spectral droplet}. We then have 
$\mathcal{S}\subset \mathcal{N}_{+,0}$.
The point process \eqref{eq-probdist2} has the following property: as 
$m,n\to+\infty$ while $n=m+\Ordo(1)$, the points will tend to accumulate on 
the set $\mathcal{S}$ with density $\Delta Q$ there. Moreover, the set
$\mathcal{S}\cap\mathcal{N}_+$ is rather regular for real-analytic $Q$, as
the Sakai theory applies (cf. \cite{HedShi}). Typically we then expect a
real-analytic boundary, with the exception of cusps and contact (or kissing) 
points. Let us refer to the set
$\mathrm{int}(\mathcal{S}\cap\mathcal{N}_{+})$ as {\em the bulk}. The results 
of \cite{ahm1}, \cite{ahm2} show that for bulk points
$z$, the local blow-up process at $z$, with coordinates $(\xi_1,\ldots\xi_n)$, 
\[
\xi_j:=m^{1/2}[\Delta Q(z)]^{1/2}(z_j-z),
\]
where $(z_1,\ldots,z_n)$ are from the process \eqref{eq-probdist2}, 
converges weakly to the translation invariant Ginibre$(\infty)$
process, as $m,n\to+\infty$ while $n=m+\ordo(1)$. 
The associated generating kernel is the reproducing kernel 
$(\xi,\eta)\mapsto\e^{\xi\bar\eta}$ of the 
Bargmann-Fock space. 
This has the flavor of a universality result.
The corresponding statement in the Gaussian unitary ensemble 
(GUE) case is the universality of the sine kernel for bulk. We observe here 
that the sine kernel is the reproducing kernel for the Paley-Wiener space 
(a subspace of $L^2(\R)$ consisting of entire functions). 
As for the two boundary points in the GUE model, the Tracy-Widom distribution 
appears, which is generated by the Airy kernel. The Airy kernel is 
reproducing for another Hilbert space of entire functions.   
This suggests that for real-analytic $Q$ and $z\in\partial\mathcal{S}\cap
\mathcal{N}_+$, there should exist a local blow-up 
\[
\xi_j:=m^\theta(z_j-z),
\]
where $\theta=\theta(z)$ is a suitable positive real, such that as 
$m,n\to+\infty$ while $n=m+\ordo(1)$, the process $(\xi_1,\ldots,\xi_n)$
would converge to a determinantal process whose generating kernel is the
reproducing kernel of a Hilbert space of entire functions.   
We verify this in the context of the Ginibre ensemble (i.e., with 
$Q(z)=|z|^2$), and identify the associated Hilbert space with a closed 
subspace of the Bargmann-Fock space characterized by slow growth in a 
half-plane. In that case, $\theta=\frac12$ as in the case of interior points. 
\medskip

\noindent{\bf The Berezin measure and the Berezin density.} 
In \cite{ahm1}, \cite{ahm2}, Ameur, Hedenmalm, and Makarov study 
the Berezin measure 
$\diff B_{m,n}^{\langle z\rangle}$ and Berezin density 
$\berd_{m,n}^{\langle z\rangle}$:
\[ 
\diff B_{m,n}^{\langle z\rangle}(w):=\berd_{m,n}^{\langle z\rangle}(w)
\diff A(w),\quad
 \berd_{m,n}^{\langle z\rangle}(w):= 
\frac{|K_{m,n}(z,w)|^2}{K_{m,n}(z,z)}\e^{-mQ(w)},
\]
which arise in the study of the Berezin transform. 
The Berezin measure is a probability measure, which makes it more stable than
the reproducing kernel $K_{m,n}$ of $\mathrm{Pol}_{m,n}$ itself as we let 
$m,n \to+\infty$. In terms of the point process \eqref{eq-probdist2}, 
$\berd_{m,n}^{\langle z\rangle}$ measures the amount of repulsion from $z$ 
caused by placing one of the points at $z$.  
For bulk points $z$, we have the convergence 
$\diff B^{\langle z\rangle}_{m,n}\to\diff\delta_z$ in the weak-star 
sense of measures, as $m,n\to+\infty$ while $n=m+\ordo(1)$. Here, 
$\diff\delta_z$ is the Dirac point mass at $z$. In fact, there is
a better result: the blow-up Berezin density 
\[
\hat{\berd}_{m,n}^{\langle z\rangle}(\xi)=\frac{1}{m \Delta Q(z)}\,
\berd^{\langle z\rangle}_{m,n}\bigg(z+ \frac{\xi}{\sqrt{m \Delta Q(z)}}\bigg) 
\]
converges to the standard Gaussian $\e^{-|\xi|^2}$. This corresponds to
the convergence of the local blow-up of the point process to the 
Ginibre$(\infty)$ process (cf. \cite{ahm1}, \cite{ahm2}).
On the other hand, for points $z$ outside the spectral droplet, i.e.,
for $z\in\C\setminus\mathcal{S}$, the Berezin 
measure $\diff B^{\langle z\rangle}_{m,n}$ converges in the weak-star sense 
of measures to harmonic measure at $z$ with respect to the exterior domain
$\C\setminus\mathcal{S}$ as $m,n\to+\infty$ while $n=m+\mathrm{o}(1)$
(see \cite{ahm1} for $Q=|z|^2$, and \cite{ahm2} for
the general result). 
\medskip

\noindent{\bf The local blow-up of the point process and the Berezin density.}
It is convenient to think of the point process \eqref{eq-probdist2} in terms
of the $k$-point intensities
\[
\det\big[\Zhe_{m,n}(z_i,z_j)\big]_{i,j=1}^{k}
=\det\big[K_{m,n}(z_i,z_j)
\e^{-\frac{1}{2}m[Q(z_i)+Q(z_j)]}\big]_{i,j=1}^{k}.
\]
We notice quickly that the intensities are unchanged if the kernel changed
to 
\[
K^\chi_{m,n}(z,w):=\chi(z)\bar\chi(w)K_{m,n}(z,w),
\]
provided that $\chi$ is measurable with $|\chi(z)|\equiv1$ 
(we can call this a ``guage transformation''). This can help in the 
asymptotical analysis of local blow-ups. For $k=2$, we get the $2$-point 
intensity
\begin{multline*}
\big\{K_{m,n}(z_1,z_1)K_{m,n}(z_2,z_2)-|K_{m,n}(z_1,z_2)|^2\big\}
\e^{-mQ(z_1)-mQ(z_2)}
\\
=K_{m,n}(z_1,z_1)\e^{-mQ(z_1)}\bigg\{K_{m,n}(z_2,z_2)\e^{-mQ(z_2)}
-\frac{|K_{m,n}(z_1,z_2)|^2}{K(z_1,z_1)}\e^{-mQ(z_2)}\bigg\},
\end{multline*}
where we recognize the Berezin density as a correction to the product of
the two $1$-point densities, where the $1$-point intensity is 
\[
K_{m,n}(z_1,z_1)\,\e^{-mQ(z_1)}.
\]
So, as far as the $2$-point intensity goes, we just need the $1$-point 
intensity and the Berezin density. Since the $1$-point intensity is just the
restriction to $z_1=z_2$ of the Berezin density, the Berezin density is all
we need to describe the $2$-point intensity. We will be a little lazy
and just work with the Berezin density in the context of local blow-ups, 
although the asymptotics of the $k$-point intensity would strictly speaking
require a little more work. So, although we state our many of our assertions 
regarding local blow-ups in terms of the Berezin density, we maintain that 
they generalize to statements about the point processes (cf. \cite{ahm2}).
\medskip

\noindent{\bf The Ginibre ensemble and its polyanalytic generalization.}
The case $Q(z)= |z|^2$ of \eqref{eq-ginQ} (or \eqref{eq-probdist2}
is known as the {\em Ginibre ensemble}. The (probability generating or 
reproducing) kernel is now particularly simple:  
\[
K_{m,n}(z,w)= m\sum_{j=0}^{n-1} \frac{(mz \bar{w})^j}{j!}. 
\] 
Here, $\mathcal{S}=\bar\D$, the closed unit disk. 
We will consider a family of generalizations of the Ginibre ensemble, 
the {\em polyanalytic Ginibre ensembles}, which are defined by the 
reproducing kernels $K_{m,n,\Rpar}$ 
of the subspaces 
\[ 
\mathrm{Pol}_{m,n,\Rpar} := \mathrm{span}\,
\big\{ z^j\overline{z}^k:\, 0\le j\le n-1,\,\,0\leq k \leq\Rpar-1\big\} 
\] 
supplied with the Hilbert space structure of $L^2(\C,\e^{-m|z|^2})$. 
The parameter value $\Rpar=1$ corresponds to the standard Ginibre process. 
In general, we now project to the $n\Rpar$-dimensional subspace of the
{\em polyanalytic polynomials}, where the degree in $\bar z$ is $\le\Rpar-1$, 
and the degree in $z$ is $\le n-1$. Note that the dimension of the subspace is
now $nq$ and not $n$. We can think of the case $q>1$  as permitting more Landau
levels than the lowest one (there are two similar models, see \cite{AIM}, 
\cite{Mou}, \cite{vas}; see also \cite{abreu}). 
We will keep the polyanalyticity degree $\Rpar$ fixed in the process, while 
we let both $m,n$ tend to infinity. We remark here that we have noticed that
Tomoyuki Shirai is interested in a similar model \cite{tshi}.

As for the point process, the points generally repel each other, but for
$\Rpar>1$, they also tend to avoid certain geometric configurations, such as
circles and lines. We have run a simulation in Figure \ref{fig-1}.
\medskip

\begin{figure}
\includegraphics[angle= 0,width=100mm]{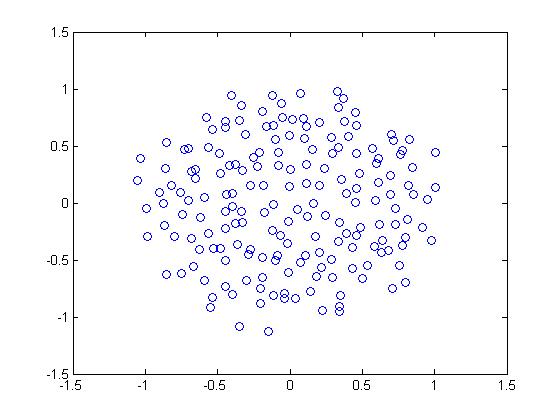} 
\caption{The polyanalytic Ginibre process with the kernel $K_{m,n,q}$ with 
$m=n=61$ and $q=3$.
The simulation is based on the algorithm by Hough, Krishnapur, Peres, and 
Vir\'{a}g \cite{hkpv}.}
\label{fig-1}
\end{figure}

\noindent{\bf Results.} Macroscopically, we find that the behavior of the 
polyanalytic Ginibre ensemble is similar to that of the Ginibre ensemble. 
If we measure this in terms of the Berezin measure, we have that as 
$m,n\to+\infty$ while $n=m+\Ordo(1)$, 
\[
\diff B_{m,n,\Rpar}^{\langle z\rangle}\to\diff\delta_z\,\,\,\text{for}\,\,\,
z\in\bar\D,\quad
\diff B_{m,n,\Rpar}^{\langle z\rangle}\to\diff\omega_z\,\,\,\text{for}\,\,\,
z\in\D^e,
\]
where $\omega_z$ is the harmonic measure of $z$ with respect to the exterior 
disk $\D^e$. Here, $K_{m,n,\Rpar}$ is the reproducing kernel for 
$\mathrm{Pol}_{m,n,\Rpar}$, and we use the notation
\[ 
\diff B_{m,n,\Rpar}^{\langle z\rangle}(w):=
\berd_{m,n\Rpar}^{\langle z\rangle}(w)\diff A(w),\quad
 \berd_{m,n,\Rpar}^{\langle z\rangle}(w):= 
\frac{|K_{m,n,\Rpar}(z,w)|^2}{K_{m,n,\Rpar}(z,z)}\e^{-mQ(w)}, 
\]
for the corresponding Berezin measure and Berezin density. Interestingly, 
the microscopic behavior of the Berezin measure 
$\diff B^{\langle z\rangle}_{m,n,\Rpar}$ is 
quite different for $\Rpar>1$ compared with the Ginibre case $\Rpar=1$. 
In terms of the blow-up Berezin density ($\Delta Q(z)=\Delta|z|^2\equiv1$ here)
\[
\hat{\berd}_{m,n,\Rpar}^{\langle z\rangle}(\xi)=m^{-1}\,
\berd^{\langle z\rangle}_{m,n,\Rpar}(z+m^{-1/2}\xi), 
\]
we have the following asymptotics as $m,n\to+\infty$ while $n=m+\Ordo(1)$:
\[
\hat{\berd}_{m,n,\Rpar}^{\langle z\rangle}(\xi)\longrightarrow\Rpar^{-1}
L^{1}_{\Rpar-1}(|\xi|^2)^2\e^{-|\xi|^2}\quad\text{for}\,\,\, z\in\D,
\]
where $L^{1}_{\Rpar-1}$ denotes the (generalized) Laguerre 
polynomial of degree $\Rpar-1$ with parameter $1$. It is well-known that 
the Laguerre polynomial $L^{1}_{\Rpar-1}$ has exactly $\Rpar-1$ strictly 
positive roots, which implies that the Berezin density will exhibit a typical 
{\em Fresnel-type ring pattern}. This resembles what happens in the 
one-dimensional GUE case, where the zero density points for the Berezin
density come from the zeros of the sine kernel. Actually, the analogy is more
than a superficial similarity. If we consider rather big values of $q$, 
and scale down to local distance $(mq)^{-1/2}$, with 
\[
\tilde{\berd}_{m,n,\Rpar}^{\langle z\rangle}(\xi')=(mq)^{-1}\,
\berd^{\langle z\rangle}_{m,n,\Rpar}\big(z+(mq)^{-1/2}\xi'\big), 
\]
then 
by the above we have, for $z\in\D$, 
\[
\hat{\berd}_{m,n,\Rpar}^{\langle z\rangle}(\xi)\longrightarrow\Rpar^{-2}
L^{1}_{\Rpar-1}(q^{-1}|\xi'|^2)^2\e^{-q^{-1}|\xi'|^2},
\]
as $m,n\to+\infty$ with $n=m+\Ordo(1)$. Next, if we let $q\to+\infty$, we
get that
\[
\Rpar^{-2}
L^{1}_{\Rpar-1}(q^{-1}|\xi'|^2)^2\e^{-q^{-1}|\xi'|^2}\longrightarrow
\bigg\{\sum_{i=0}^{+\infty}\frac{(-1)^i|\xi'|^{2i}}{i!(i+1)!}\bigg\}^2
=\frac{J_1(2|\xi'|)^2}{|\xi'|^2},
\]
where $J_1$ is the standard Bessel function. The identity
\[
\frac{J_1(2|\xi'|)}{|\xi'|}=\int_\D\e^{2\imag\re[\xi'\bar\zeta]}\diff A(\zeta)
\]
shows that we are dealing with a planar analogue of the sine kernel 
(the sine kernel is the Fourier transform of the characteristic function of
the interval $[-1,1]$, the one-dimensional unit ball).

We also investigate the local behaviour of the Berezin transform when 
$|z|=1$, i.e., when $z$ is on the boundary of the bulk. Using the same blow-up 
scale as with an interior point, we show that the blow-up Berezin density
$\hat{\berd}^{\langle z\rangle}_{m,n,\Rpar}(\xi)$ tends to a limit which can 
be expressed in terms of Hermite polynomials (see Theorem 
\ref{thm-5.10}). Here, too, there is a ring-like pattern in the interior 
direction, but it is not so pronounced as it is for interior points (the bulk).
We express the $1$-point intensity near a boundary point in terms of 
a sum of squares of Hermite polynomials. The Wigner semi-circle law 
then gives the asymptotic behavior of the $1$-point function, which tells us
the intensity of the process. We find that for big $q$, but much bigger
$m,n$ with $m=n+\Ordo(1)$, the $1$-point function is nontrivial in the annulus 
\[
1-q^{1/2}m^{-1/2}\le|z|\le1+q^{1/2}m^{-1/2};
\]
inside it is essentially constant, and outside it more or less vanishes.
Near the outward boundary of the annulus at the scale $(mq)^{-1/2}$,
we expect the statistics of the point process to be related with the 
well-known Airy process.
\medskip
 
\noindent{\bf Lifting to two complex variables.} Analogous results to 
\cite{ahm1} are obtained on complex manifolds by Berman \cite{berman1}. 
We note that the polyanalytic Ginibre processes can also be viewed as 
processes in $\mathbb{C}^2$ with the rather singular weight  
$\e^{-|z_1|^2}\delta_0(z_1-\bar z_2)$, where $\delta_0$ is unit point mass 
at $0$. Berman considers reproducing kernels of polynomial subspaces as the 
total degree of the polynomials tends to infinity. In contrast, here we 
discuss the case where one variable has bounded degree and the degree of 
the other variable goes to infinity. 
\medskip

\noindent{\bf The polyanalytic Ginibre ensemble and weighted interpolation.} 
It is well known in the theory of random matrices that 
\[ 
\frac{1}{n!} \det[K_{m,n,1}(z_i, z_j)]_{i,j=1}^n 
= \frac{1}{Z_{m,n,1}} |\Delta(z_1,...,z_n)|^2, 
\]
where $\Delta(z_1,...,z_n)= \Pi_{i,j:i<j} (z_j-z_i)$ is the van der Monde  
determinant. 
A point configuration in a compact set which maximizes the van der Monde 
determinant is known to 
 be a good 
choice of nodes for Lagrange interpolation \cite{st}. Instead of considering 
points confined to a compact set, one can add a weight to prevent the points 
from going off to infinity. This leads to the same expression which arises 
in the context of random eigenvalues. 

We turn to the polyanalytic case. One shows that 
\[ 
\frac{1}{(n\Rpar)!} \det[K_{m,n,\Rpar}(z_i,z_j)]_{i,j=1}^{n\Rpar} 
= \frac{1}{Z_{m,n,\Rpar}}| \Delta_\Rpar(z_1,...,z_{n\Rpar})|^2, \]
where $Z_{m,n,\Rpar}$ is a normalization constant  
\[ 
Z_{m,n,\Rpar}:=
\int_{\C^{n\Rpar}}|\Delta_\Rpar(z_1,...,z_{n\Rpar})|^2
\e^{-m \sum_{j=1}^{n\Rpar}Q(z_j)}
\diff A(z_1) \dots \diff A(z_{n\Rpar})
\]
and 
\[ 
\Delta_\Rpar(z_1,...,z_{n\Rpar}) = 
\mathrm{det} \left( 
\begin{array}{ccc}
1 & \dots  & 1 
\\
z_1 & \dots & z_{n\Rpar} 
\\
\vdots & \dots & \vdots 
\\
z_1^{n-1} & \dots & z_{n\Rpar}^{n-1} 
\\
\bar{z}_1&\dots &  \bar{z}_{n\Rpar}
\\
\bar{z}_1z_1&\dots &  \bar{z}_{n\Rpar}z_{n\Rpar}
\\ 
\vdots & \dots & \vdots 
\\
\bar{z}_1^{\Rpar-1}&\dots &  \bar{z}_{n\Rpar}^{\Rpar-1}
\\
\vdots & \dots & \vdots 
\\
\bar{z}_1^{\Rpar-1}z_1^{n-1}&\dots &  \bar{z}_{n\Rpar}^{\Rpar-1}
z_{n\Rpar}^{n-1}
\\ 
\end{array} 
\right).
\] 
So, $\Delta_\Rpar(z_1,...,z_{n\Rpar})$ is the polyanalytic analogue 
of the van der Monde determinant. As in the case $\Rpar=1$ (which gives 
the usual van der Monde determinant), the expression 
$|\Delta_\Rpar(z_1,...,z_{n\Rpar})|^2$ measures how good the configuration 
is for Lagrange interpolation by polyanalytic polynomials. So, the 
polyanalytic Ginibre ensemble is a way to produce random Lagrange
interpolation sets, using the Gaussian weight for confinement.
\medskip
 
\noindent{\bf Further results and open problems.}
In \cite{rv} and \cite{ahm2}, the authors showed that the fluctuations of 
eigenvalues of random normal matrices tend to Gaussian free field.
The fluctuations of the polyanalytic Ginibre process will be discussed in 
a separate paper -- the limit is again the Gaussian free field, but the 
variance depends on the degree of polyanalyticity. 

One could naturally address all the questions discussed in this article with 
a more general weight $Q$. We conjecture that the spectral droplet will be the 
same as in the analytic case. It is also likely that the blow-up of the 
Berezin density at a bulk point $z$ will have a universal limit, which we 
here computed to be 
$\Rpar^{-1}L^{\langle1\rangle}_{\Rpar-1}(|\xi|^2)^2\e^{-|\xi|^2}$.   

\section{Polyanalytic Bargmann-Fock spaces}
\label{sec-polybargmann}

\noindent{\bf An orthogonal basis.} 
We will consider the Bargmann-Fock space $A^2_{m,\Rpar}(\C)$ of poly-analytic 
functions of degree $\le \Rpar-1$, i.e., functions of the form 
\[
f(z) = \sum_{r=0}^{\Rpar-1} \bar{z}^r f_r(z),
\]
where all the components $f_r$ are entire, subject to the norm integrability 
condition 
\[
\|f\|^2_{A^2_{m,\Rpar}(\C)}:=
\|f\|^2_{L^2(\C,\e^{-m|z|^2})}
=\int_{\C} |f(z)|^2 \e^{-m|z|^2} \diff A(z)<+\infty.
\] 
Here, as always, $m>0$. We note that $A^2_{1,1}(\C)$ is the standard 
Bargmann-Fock space. As before, let $\mathrm{Pol}_{m,n,\Rpar}$ be the 
closed subspace of $A^2_{m,\Rpar}(\C)$, defined by the condition that all the 
components $f_r$ are polynomials of degree less or equal to $n-1$.
Moreover, let $K_{m,\Rpar}$ and $K_{m,n,\Rpar}$ denote the reproducing kernels 
for $A^2_{m,\Rpar}(\C)$ and $\mathrm{Pol}_{m,n,\Rpar}$, respectively. 
We will be concerned with the asymptotic behaviour of the kernel 
$K_{m,n,\Rpar}$, as $m,n\to+\infty$ while $n=m+\Ordo(1)$. 

Bargmann-Fock-spaces of polyanalytic functions have been considered in, e.g.,
\cite{vas}, where the reproducing kernels and orthonormal bases were 
identified. We will attempt to supply a self-contained account of these 
basic matters.   

To begin with, we identify an orthonormal basis for the space 
$\mathrm{Pol}_{m,n,\Rpar}$. Here, and later as well, we will need some 
standard properties of the classical orthogonal polynomials. 
For details, we refer the reader to \cite{orthopo}. We need the 
generalized Laguerre polynomials ($(x)_j:=x(x+1)\cdots(x+j-1)$ is the 
Pochhammer symbol)
\[
L^\alpha_k(x):=\sum_{i=0}^{k}(-1)^i\binom{k+\alpha}{k-i}\frac{x^i}{i!}
=\sum_{i=0}^{k}(-1)^i\frac{(\alpha+i+1)_{k-i}}{i!(k-i)!}x^i.
\]

\begin{prop} \label{on-kanta}
For $\Rpar\le n$, the following functions form an orthonormal basis for 
$\mathrm{Pol}_{m,n,\Rpar}$ ($i,r,j,k$ are all integer parameters): 
\begin{align*} 
e_{i,r}^{1}(z)&:=\sqrt{\frac{r!}{(r+i)!}} \,m^{(i+1)/2}z^i\, L_r^{i}(m|z|^2), 
\quad 0\leq i\leq n-r-1,\,\,\,0\leq r\leq\Rpar-1, 
\\
e_{j,k}^{2}(z)&:=\sqrt{\frac{j!}{(j+k)!}} 
\,m^{(k+1)/2}\bar{z}^k\, L^{k}_{j}(m|z|^2),
\quad 0\leq j\leq\Rpar-k-1,\,\,\,1\leq k\leq\Rpar-1. 
\end{align*}
\end{prop}
\begin{proof}
Clearly, all the above functions belong to the space 
$\mathrm{Pol}_{m,n,\Rpar}$. Also, all the functions $e^1_{i,r}$ are orthogonal
to each of the functions $e^2_{j,k}$, for the indicated ranges of the indices, 
as can be seen by integrating along circles using polar coordinates.
Next, we show that the functions $e^1_{i,r}$ form an orthonormal set.  
Any two functions there having different parameter $i$ are orthogonal, again
by integrating along circles using polar coordinates. 
So, we fix $i$, and pick two indices $r_1,r_2$. We compute the inner product 
of two such functions:
\begin{multline}
\int_\C e^1_{i,r_1}(z)\overline{e^1_{i,r_2}(z)}\e^{-m|z|^2} \diff A(z)
\\
=\sqrt{\frac{r_1!}{(r_1+i)!}} \sqrt{\frac{r_2!}{(r_2+i)!}}m^{i+1}
\int_{\C}
|z|^{2i}L_{r_1}^{i}(m|z|^2)L_{r_2}^{i}(m|z|^2) \e^{-m|z|^2} 
\diff A(z) \notag \\
=\sqrt{\frac{r_1!r_2!}{(r_1+i)!(r_2+i)!}} \int_{0}^{\infty} 
L_{r_1}^{i}(t) 
L_{r_2}^{i}(t)\,t^i\e^{-t} \diff t = \delta_{r_1,r_2},
\end{multline}
where the delta is in Kronecker's sense. In a similar fashion, 
the functions $e^2_{j,k}$ form an orthonormal set. 
So, the functions $e^1_{i,r},e^2_{j,k}$ together form an 
orthonormal set. Next, the dimension of the span equals the total number of
vectors, which we calculate to $n\Rpar$, which equals the known dimension 
of the space $\mathrm{Pol}_{m,n,\Rpar}$. The proof is complete. 
\end{proof}
\medskip

\noindent{\bf The reproducing kernel.}
For $\Rpar\leq n$, we conclude that the reproducing kernel of 
$\mathrm{Pol}_{m,n,\Rpar}$ equals 
\begin{multline}
\label{eq-Kmnq}
K_{m,n,\Rpar} (z,w) = m\sum_{r=0}^{\Rpar-1} 
\sum_{i=0}^{n-r-1} \frac{r!}{(r+i)!} 
(mz\bar{w})^i L_r^{i}(m|z|^2)L_r^{i}(m|w|^2) 
\\
+ m\sum_{j=0}^{\Rpar-2}\sum_{k=1}^{\Rpar-j-1} \frac{j!}{(k+j)!} 
(m\overline{z}w)^k L^{k}_{j}(m|z|^2)L^{k}_{j}(m|w|^2).
\end{multline}
We note that by plugging in $n=+\infty$ in Proposition \ref{on-kanta}
we get an orthonormal basis for the space $A^2_{m,\Rpar}(\C)$. 
It follows that the same procedure of plugging in $n=+\infty$ in the 
above expression for $K_{m,n,\Rpar}$ gives us $K_{m,\Rpar}$, the reproducing 
kernel for $A^2_{m,\Rpar}(\C)$.  What is probably less obvious is that 
$K_{m,\Rpar}$ may be written in a much simpler form (this representation is,
however, known; compare with \cite{AIM}, \cite{Mou}).

\begin{prop}
\label{prop-repkern}
We have that
\begin{multline*}
K_{m,\Rpar} (z,w) = m\sum_{r=0}^{\Rpar-1}\sum_{i=0}^{\infty} 
\frac{r!}{(r+i)!}(mz \bar{w})^i L_r^{i}(m|z|^2)L_r^{i}(m|w|^2) \notag 
\\
+m\sum_{j=0}^{\Rpar-2} \sum_{k=1}^{\Rpar-j-1} \frac{j!}{(j+k)!} 
(m\bar{z}w)^k L^{k}_{j}(m|z|^2)L^{k}_{j}(m|w|^2) \notag 
=m\,L^{1}_{\Rpar-1}(m|z-w|^2)\e^{m z \bar{w}}.\notag
\end{multline*}
\end{prop}

\begin{proof}
It is clear from Proposition \ref{on-kanta} that the double sum expression
equals $K_{m,\Rpar}(z,w)$, so the only thing that needs attention is 
the last equality.
We first do the case $w=0$. Many of the terms in the double sum vanish,
so that we are left with
\begin{equation}
K_{m,\Rpar} (z,0)=m\sum_{r=0}^{\Rpar-1}L_r^{0}(m|z|^2)L_r^{0}(0),  
\notag
\end{equation}
It is well-known that $L_r^{0}(0)=1$ and that 
$\sum_{r=0}^{\Rpar-1} L^{\alpha}_r=L^{\alpha+1}_{\Rpar-1}$, so that 
the above reduces to
\begin{equation}
K_{m,\Rpar} (z,0) = m\, L^{1}_{\Rpar-1}(m|z|^2).  
\notag
\end{equation}
We turn to general $w \in\C$. It is well-known that the transformation
\[
\Tope_w[f](z):=\e^{-\frac12 m|w|^2-mz\bar w}f(z+w)
\]
acts unitarily on $A^2_{m,\Rpar}(\C)$; its adjoint is $\Tope_w^*=\Tope_{-w}$. 
By the reproducing property of $K_{m,\Rpar}(z,0)$, we have, for $f\in 
A^2_{m,\Rpar}(\C)$,
\begin{multline*}
\e^{-\frac12m|w|^2}f(w)=\Tope_w[f](0)=
\int_\C \Tope_w f(z)\overline{K_{m,\Rpar}(z,0)}\e^{-m|z|^2}\diff A(z)
\\
=\int_\C f(z)\overline{\Tope_w^*[K_{m,\Rpar}(\cdot,0)](z)}\,
\e^{-m|z|^2}\diff A(z)
\\
=m\e^{-\frac12m|w|^2}
\int_\C f(z)\,L^{1}_\Rpar(m|z-w|^2)\e^{m\bar z w}\e^{-m|z|^2}\diff A(z),
\end{multline*}
The claim that $K_{m,\Rpar}(z,w)= m\,L^{1}_{\Rpar-1}(m|z-w|^2) 
\e^{mz\bar{w}}$ now follows immediately.
\end{proof}
\medskip

\noindent{\bf Szeg\"o asymptotics.}
We shall show that the kernel $K_{m,n,\Rpar}(z,w)$ 
is approximated well by $K_{m,\Rpar}(z,w)$, provided that $z,w\in\D$ are 
rather close to one another. It will be instrumental to study the asymptotic 
behavior of the partial sums 
\[
E_{k}(\zeta)=\sum_{j=0}^{k}\frac{1}{j!} \zeta^j
\]
of the Taylor expansion of the exponential function. In \cite{szego}, 
Szeg\"o showed that
\begin{align} 
\label{szego1}
\frac{E_k(k\zeta)}{\e^{k\zeta}}= 1-\frac{1}{\sqrt{2 \pi k}}
(\zeta\e^{1-\zeta})^k \frac{\zeta}{1-\zeta}
(1+\epsilon^{1}_k(\zeta)),
\end{align}
provided that $|\zeta\e^{1-\zeta}|<1$ and $|\zeta|<1$, whereas
\begin{align} 
\label{szego2}
\frac{E_k(k\zeta)}{\e^{k\zeta}}= \frac{1}{\sqrt{2 \pi k}}(\zeta
\e^{1-\zeta})^k \frac{\zeta}{\zeta-1}
(1+\epsilon^{2}_k(\zeta)),
\end{align}
provided that $|\zeta\e^{1-\zeta}|<1$ and $|\zeta|>1$. Here, we have the 
convergence
\[
\lim_{k\to+\infty}\epsilon^{1}_k(z)=\lim_{k\to+\infty} \epsilon^{2}_k(z)=0
\]
uniformly on compact subsets of the respective domains. As for \eqref{szego2},
the uniform convergence $\epsilon^{2}_k(\zeta)\to0$  as $k\to+\infty$ holds 
also in certain unbounded subdomains; in particular, along the 
real line, we have uniform convergence on all intervals $[a,+\infty[$ with 
$a>1$.
\medskip

\noindent{\bf An elementary estimate of Laguerre polynomials.}
The following elementary estimate of generalized Laguerre polynomials will 
prove useful.
 
\begin{lem} 
\label{laguerreestimaatti}
%
Suppose $\alpha$ is a positive real. Then, for $k=0,1,2,\ldots$, we have that 
\[
|L^\alpha_k(x)|\le \frac{1}{k!}(x+\alpha+k)^k,\qquad x\in
[0,+\infty[.
\]
Moreover, with $\beta:=\alpha+2k-2+2\sqrt{\frac14+(k-1)(k+\alpha-1)}$,
we have that
\[
\frac{1}{k!}(x-\beta)^k\le (-1)^k L_k^{\alpha}(x)\le\frac{1}{k!}x^k,\qquad
x\in[\beta,+\infty[.
\]
Actually, the related inequality $|L_k^{\alpha}(x)|\le\frac{1}{k!}x^k$ holds 
for all $x\in[\frac12\beta,+\infty[$. 
\end{lem}

\begin{proof}
We begin with the estimate ($0\le i\le k$ is assumed)
\[
(\alpha+i+1)_{k-i}\le(\alpha+k)^j.
\]
Next, we note that for $x\ge0$ and $k=0,1,2,\ldots$, we have
\[
L^\alpha_k(x)=\sum_{i=0}^k(-1)^k\frac{(\alpha+i+1)_{k-i}}{i!(k-i)!}
x^i=\sum_{i\,\,\mathrm{even}}\frac{(\alpha+i+1)_{k-i}}{i!(k-i)!}
x^i-\sum_{i\,\,\mathrm{odd}}\frac{(\alpha+i+1)_{k-i}}{i!(k-i)!}
x^i,
\]
and in view of the above estimate,
\begin{equation*}
\sum_{i\,\,\mathrm{even}}\frac{(\alpha+i+1)_{k-i}}{i!(k-i)!}
\le \sum_{i\,\,\mathrm{even}}\frac{(\alpha+k)^{k-i}}{i!(k-i)!}
=\frac{1}{2\,k!}\big((\alpha+k+x)^k+(\alpha+k-x)^k\big),
\end{equation*}
and, analogously,
\begin{equation*}
\sum_{i\,\,\mathrm{odd}}\frac{(\alpha+i+1)_{k-i}}{i!(k-i)!}
\le \sum_{i\,\,\mathrm{odd}}\frac{(\alpha+k)^{k-i}}{i!(k-i)!}
=\frac{1}{2\,k!}\big((\alpha+k+x)^k-(\alpha+k-x)^k\big).
\end{equation*}
By discarding alternatively the even or odd contribution, we arrive at
\[
|L^\alpha_k(x)|\le\frac{1}{2\,k!}\big((x+\alpha+k)^k+|\alpha+k-x|^k\big),
\]
which is slightly stronger than the first estimate.

It is well-known that $L_k^\alpha$ is a polynomial of degree $k$ all of
whose zeros are distinct and real, and they all fall in the interval 
$]0,\beta[$ (see 18.16.13 of \cite{orthopo}). The second estimate follows 
immediately from this. 
\end{proof}
\medskip

\noindent{\bf Approximation of the polynomial polyanalytic reproducing 
kernel.} We estimate the difference $K_{m,n,\Rpar}-K_{m,\Rpar}$.

\begin{prop} 
\label{kernelestimate}
Let $z,w\in\C$ be such that $|zw|\le\theta_1<1$ and 
$|zw|\e^{1-|zw|}\le\theta_2<1$. Let $M$ be a positive real number. 
Then, as $m,n\to+\infty$ while $|m-n|\le M$,
\begin{equation*} 
|K_{m,n,\Rpar}(z,w)-K_{m,\Rpar}(z,w)| \,\e^{-m|zw|}
\le C m^{\Rpar}\theta_2^m(1-\theta_1)^{-2}[1+|z|^{2\Rpar}+|w|^{2\Rpar}], 
\end{equation*}
where the constant $C$ depends on $\Rpar$ and $M$.
\end{prop}

\begin{proof}
In view of \eqref{eq-Kmnq} and Proposition \ref{prop-repkern}, we have that
for $\Rpar\le n$,
\[
K_{m,\Rpar}(z,w)-K_{m,n,\Rpar}(z,w)=m\sum_{r=0}^{\Rpar-1}\sum_{i=n-r}^{+\infty}
\frac{r!}{(r+i)!}(mz \bar{w})^i L_r^{i}(m|z|^2)L_r^{i}(m|w|^2). 
\]
By Lemma \ref{laguerreestimaatti}, we get that
\begin{multline}
\label{eq-interm1}
\bigg| \sum^{+\infty}_{i=n-r} 
\frac{r!}{(r+i)!}(mz \bar{w})^i L_r^{i}(m|z|^2)L_r^{i}(m|w|^2)\bigg| 
\\
\le \sum^{+\infty}_{i=n-r}\frac{1}{r!(r+i)!}(m|zw|)^i(m|z|^2+i+r)^r
(m|w|^2+i+r)^r 
\\  
=\frac1{r!}\sum^{+\infty}_{j=0}\frac{1}{(n+j)!}(m|zw|)^{n-r+j}
(m|z|^2+n+j)^r(m|w|^2+n+j)^r
\\
\le 
\frac{4^{r-1}}{r!}\sum^{+\infty}_{j=0}\frac{1}{(n+j)!}(m|zw|)^{n-r+j}
[(m|z|^2)^r+(n+j)^r][(m|w|^2)^r+(n+j)^r].
\end{multline}
For $r$ confined to $0\le r\le\Rpar-1$ and for big $n$, say $n\ge n_0(\Rpar)$,
we have
\[
\frac{(n+j)^{2r}}{(n+j)!}\le \frac{2}{(n+j-2r)!}\quad\text{and}\quad
\frac{(n+j)^{r}}{(n+j)!}\le \frac{2}{(n+j-r)!},
\]
so that if we use the notation 
\[
R_k(z)=\sum_{j=k+1}^{+\infty}\frac1{j!}z^j=\e^z-E_k(z),
\]
we obtain from \eqref{eq-interm1} that for $0\le r\le\Rpar-1$ and
$n\ge n_0(\Rpar)$,
\begin{multline}
\label{eq-interm2}
\bigg| \sum^{+\infty}_{i=n-r} 
\frac{r!}{(r+i)!}(mz \bar{w})^i L_r^{i}(m|z|^2)L_r^{i}(m|w|^2)\bigg| 
\\
\le
\frac{4^{r-1}m^r}{r!}\Big[
|zw|^r R_{n-1}(m|zw|)+
2(|z|^{2r}+|w|^{2r}) R_{n-r-1}(m|zw|)
+2|zw|^r R_{n-2r-1}(m|zw|)\Big].
\end{multline}
Next, we see from Szeg\"o's asymptotical expansion \eqref{szego1} that
for large $k,l$ with $l=k+\Ordo(1)$, we have
\begin{align} 
\label{szego3}
\frac{R_k(l\zeta)}{\e^{l\zeta}}= (2 \pi k)^{-1/2}(l/k)^{k}
\frac{\zeta}{k/l-\zeta}(z\e^{1-l\zeta/k})^k(1+\epsilon^{1}_k(l\zeta/k)),
\end{align}
as $k,l\to+\infty$ and $l=k+\Ordo(1)$. 
By a careful application of \eqref{szego3} to \eqref{eq-interm2}, and summing
over $0\le r\le \Rpar-1$, the assertion of the proposition follows. 
\end{proof}

\begin{cor} 
\label{cor-kapprox}
For $z,w\in\D$, let $1-|zw|\ge\tau>0$. Then, as 
$m,n\to+\infty$, while $|m-n|\le M$,
\begin{equation*} 
K_{m,n,\Rpar}(z,w)=K_{m,\Rpar}(z,w)+\Ordo(\e^{-\frac12m\tau^2}
\e^{m|zw|}), 
\end{equation*}
where the ``$\Ordo$'' constant depends on $\tau$, $\Rpar$, and $M$. 
\end{cor}

\begin{proof}
A Taylor series expansion of the logarithm gives that
\[
t\,\e^{1-t}<\e^{-\frac12(1-t)^2},\qquad 0\le t<1,
\] 
and, together with the fact that the exponential function grows faster than
any given power, the assertion follows from Proposition \ref{kernelestimate}. 
\end{proof}

As we shall see, Proposition \ref{kernelestimate} implies that the Berezin 
density 
\[
\berd^{\langle z\rangle}_{m,n,\Rpar}(w)= 
\frac{|K_{m,n,\Rpar}(z,w)|^2 }{K_{m,n,\Rpar}(z,z)}\e^{-m|w|^2}
\] 
behaves locally near $z$ like the Berezin density 
\begin{multline}
\label{eq-interm3}
\berd^{\langle z\rangle}_{m,\Rpar}(w)= 
\frac{| K_{m,\Rpar}(z,w)|^2 }{K_{m,\Rpar}(z,z)}\e^{-m|w|^2}=
\frac{m^2L^1_{\Rpar-1}(m|z-w|^2)^2\e^{2m\re z\bar w}}
{mL^1_{q-1}(0)\,\e^{m|z|^2}}\,\e^{-m|w|^2}
\\
=\frac{m}{\Rpar}L^1_{\Rpar-1}(m|z-w|^2)^2\,\e^{-m|z-w|^2}.
\end{multline}
To make this precise, we recall the definition of the blow-up Berezin density: 
\[
\hat{\berd}^{\langle z\rangle}_{m,n,\Rpar}(\xi) = 
m^{-1}\berd^{\langle z\rangle}_{m,n,\Rpar}(z+m^{-1/2}\xi).
\]

\begin{prop}
\label{prop-berconv}
Fix $z\in\D$. Then 
\[
\int_{\mathbb{C}} \big|\hat{\berd}^{\langle z\rangle}_{m,n,\Rpar}(\xi)- 
\Rpar^{-1}L^{1}_{\Rpar-1}(|\xi|^2)^2\e^{-|\xi|^2} \big| \diff A(\xi) 
\longrightarrow 0,
\]
as $m,n\to+\infty$ while $n=m+\Ordo(1)$.
\end{prop}

\begin{proof}
By \eqref{eq-interm3}, we have
\[
\hat{\berd}^{\langle z\rangle}_{m,\Rpar}(\xi) = 
m^{-1}\berd^{\langle z\rangle}_{m,\Rpar}(z+m^{-1/2}\xi)=
\Rpar^{-1}L^{1}_{\Rpar-1}(|\xi|^2)^2\e^{-|\xi|^2},
\]
so the comparison is with the blow-up Berezin density for $K_{m,\Rpar}$.
We write $w:=z+m^{-1/2}\xi$; since $z\in\D$ is fixed, we have 
$1-|z|^2\ge\tau>0$ for some small $\tau$, and we suppose that $w$ is close
to $z$ so that $1-|zw|\ge\tau>0$ as well.
Then, by Proposition \ref{prop-repkern} and Corollary \ref{cor-kapprox}, 
\begin{multline*}
\berd^{\langle z\rangle}_{m,n,\Rpar}(w)= 
\frac{|K_{m,n,\Rpar}(z,w)|^2}{K_{m,n,\Rpar}(z,z)}\e^{-m|w|^2} 
=\frac{|K_{m,\Rpar}(z,w)+\Ordo(\e^{-\frac12 m\tau^2}\e^{m|zw|}) 
|^2}{K_{m,\Rpar}(z,z)+ \Ordo(\e^{-\frac12 m\tau^2}\e^{m|z|^2})}
\e^{-m|w|^2}  
\\
=\frac{\big|mL^{1}_{\Rpar-1}(m|z-w|^2)
\e^{mz\bar{w}} + \Ordo(\e^{-\frac12 m\tau^2}\e^{m|zw|})  
\big|^2}{(m\Rpar+\Ordo(\e^{-\frac12 m\tau^2}))\,\e^{m|z|^2}}\e^{-m|w|^2}, 
\end{multline*}
so that if we we expand the square using that
\[
|L^1_{\Rpar-1}(m|z-w|^2)|=\Ordo(m^{\Rpar-1}),
\]
where the ``O'' depends only on $q$ (this follows from Lemma 
\ref{laguerreestimaatti}), and simplify the expression, we arrive at
\[
\berd^{\langle z\rangle}_{m,n,\Rpar}(w)=
\frac{m}{\Rpar}L^{1}_{\Rpar-1}(m|z-w|^2)^2
\e^{-m|z-w|^2} + 
\Ordo(m^{2\Rpar-2}\e^{-\frac12 m\tau^2}\e^{-m(|z|-|w|)^2}),    
\]
which immediate gives that
\[
\berd^{\langle z\rangle}_{m,n,\Rpar}(w)=
\frac{m}{\Rpar}L^{1}_{\Rpar-1}(m|z-w|^2)^2
\e^{-m|z-w|^2} + 
\Ordo(m^{2\Rpar-2}\e^{-\frac12 m\tau^2}).    
\]
The corresponding blow-up Berezin density then has
\begin{equation}
\label{eq-berasympt}
\hat\berd^{\langle z\rangle}_{m,n,\Rpar}(\xi)
=m^{-1}\berd^{\langle z\rangle}_{m,n,\Rpar}(z+m^{-1/2}\xi)
=\Rpar^{-1}L^{1}_{\Rpar-1}(|\xi|^2)^2
\e^{-|\xi|^2} + 
\Ordo(m^{2\Rpar-3}\e^{-\frac12 m\tau^2}).    
\end{equation}
As exponentials grow faster than polynomials, the error terms is negligible
for big $m$. For fixed $z\in\D$, the requirement on $\xi$ so that $1-|zw|\ge
\tau$ for some fixed $\tau>0$ is fulfilled for big $m$ if, say, 
$|\xi|\le\log m$ is required. So, \eqref{eq-berasympt} has the immediate
consequence that
\begin{equation}
\label{eq-berasympt2}
\int_{\D(0,\log m)}
\big|\hat\berd^{\langle z\rangle}_{m,n,\Rpar}(\xi)
-\Rpar^{-1}L^{1}_{\Rpar-1}(|\xi|^2)^2
\e^{-|\xi|^2}\big|\,\diff A(\xi)=
\Ordo\big(m^{2\Rpar-3}(\log m)^2\e^{-\frac12 m\tau^2}\big),    
\end{equation}
where more generally $\D(z_0,\rho)$ denotes the open disk of radius $\rho$ 
about $z_0$. Since the associated blow-up Berezin measures
\[
\diff\hat B^{\langle z\rangle}_{m,n,\Rpar}(\xi):=
\hat\berd^{\langle z\rangle}_{m,n,\Rpar}(w)
\diff A(\xi),\qquad \diff\hat B^{\langle z\rangle}_{m,\Rpar}(\xi):=
\Rpar^{-1}L^1_{\Rpar-1}(|\xi|^2)^2\e^{-|\xi|^2}\diff A(\xi),
\]
are both probability measures, the assertion of the proposition follows
from \eqref{eq-berasympt2} once it is noted that the right-hand side of
\eqref{eq-berasympt2} tends to $0$ as $m\to+\infty$. 
\end{proof}

We note that the claimed convergence $\diff B^{\langle z\rangle}_{m,n\Rpar}\to
\diff\delta_z$ for $z\in\D$ is an immediate consequence of Proposition 
\ref{prop-berconv}.

\section{Berezin density asymptotics for an 
exterior point}

\noindent{\bf Convergence to harmonic measure.}  
We show that the Berezin measures have the convergence
$\diff B^{\langle z\rangle}_{m,n,\Rpar}\to\diff\omega_z$ for $z\in\D^e$ as
$m,n\to+\infty$ while $n=m+\Ordo(1)$. Here, $\omega_z$ is harmonic measure 
with respect to the point $z$ and the exterior disk $\mathbb{D}^e$.
\medskip

\noindent{\bf Concentration of the Berezin mass.}  
We first study the concentration of the Berezin measure to neighborhoods of 
the closed unit disk $\bar\D$. We recall the standard notation $\D(z_0,\rho)$
for the open disk of radius $\rho$ centered at $z_0$. 

\begin{lem} Suppose $z\in\D^e$ and that $\rho>1$. Then 
\begin{equation}
\int_{\C\setminus\D(0,\rho)}\berd^{\langle z\rangle}_{m,n,\Rpar}(w)\,
\diff A(w) \to 0,
\end{equation}
as $m,n\to+\infty$ while $n\leq m+\Ordo(1)$. 
\end{lem}

\begin{proof}
Let $w\in\C\setminus\D(0,\rho)$. We note that 
\begin{equation}
n-1\le m+\Ordo(1)\le 2m
\label{eq-triv1}
\end{equation}
for sufficiently big $m$, and so, by Lemma \ref{laguerreestimaatti},
\begin{multline} 
\label{isonosanestimaatti}
\sum_{i=0}^{n-r-1} \frac{r!}{(r+i)!} 
\big|(mz\bar{w})^iL_r^{i}(m|z|^2)L_r^{i}(m|w|^2)\big| 
\\
\le\sum_{i=0}^{n-r-1} \frac{1}{r!(r+i)!} 
(m|zw|)^i(m|z|^2+i+r)^r(m|w|^2+i+r)^r 
\\
\leq \frac{9^{r}}{r!}\sum_{i=0}^{n-r-1} \frac{1}{(r+i)!}(m|zw|)^{i+2r}
\le\frac{9^{r}}{r!}(m|zw|)^{r}\,E_{n-1}(m|zw|). 
\end{multline}
If we plug in $n=\Rpar$ in \eqref{isonosanestimaatti}, we obtain  
\begin{equation} 
\label{pienenosanestimaatti}
\sum_{k=1}^{\Rpar-j-1} \frac{j!}{(j+k)!} 
\big|(m\bar{z}w)^k L^{k}_{j}(m|z|^2)L^{k}_{j}(m|w|^2)\big|  
\leq \frac{9^j}{j!}(m|zw|)^j E_{\Rpar-1}(m|zw|).
\end{equation} 
We may restrict to $\Rpar\le n$; after, we are considering the limit as 
$n\to+\infty$.
As $E_{q-1}\le E_{n-1}$ on the positive half-axis for $q\le n$, 
an application of \eqref{isonosanestimaatti} and \eqref{pienenosanestimaatti}
to \eqref{eq-Kmnq} gives
\begin{equation}
\label{isonosanestimaatti4}
|K_{m,n,\Rpar}(z,w)| 
\leq 2m\e^9(m|zw|)^{\Rpar-1}E_{n-1}(m|zw|).
\end{equation}
Next, we observe that $\mathrm{Pol}_{m,n,1}\subset \mathrm{Pol}_{m,n,\Rpar}$
(since $q\ge1$), which implies that 
\[
K_{m,n,\Rpar}(z,z) \geq K_{m,n,1}(z,z)=mE_{n-1}(m|z|^2).
\]
We conclude that the Berezin density may be estimated as follows:
\begin{align} 
\label{suttuvertausanalyyttiseen}
\berd^{\langle z\rangle}_{m,n,\Rpar}(w) \leq  
4m\,\e^{18}(m|zw|)^{2\Rpar-2}\frac{E_{n-1}(m|zw|)^2}{E_{n-1}(m|z|^2)}
\e^{-m|w|^2}.
\end{align}
Finally, we see from Szeg\"o's asymptotical expansion \eqref{szego2} that
\begin{align} 
\label{szego4}
\frac{E_k(l\zeta)}{\e^{l\zeta}}= (2 \pi k)^{-1/2}(l/k)^{k}
\frac{\zeta}{z-k/l}(\zeta\e^{1-l\zeta/k})^k(1+\epsilon^{2}_k(l\zeta/k)).
\end{align}
$k\le l+\Ordo(1)$, where the convergence 
$\epsilon^{2}_k(l\zeta/k)\to0$ is uniform if $\zeta$ is real with 
$\zeta\ge a$ for some fixed $a>1$. This leads to
\[
\frac{E_{n-1}(m|zw|)^2}{E_{n-1}(m|z|^2)}=[2\pi(n-1)]^{-1/2}[(n-1)/m]^{n-1}
\e^{n-1}|w|^{2n-2}(1+\ordo(1)),
\]
where the ``$\ordo$'' term is uniform in the convergence. As we implement this 
estimate in \eqref{suttuvertausanalyyttiseen}, and integrate over $\C\setminus
\D(0,\rho)$, the claim follows. 
\end{proof}
\medskip

\noindent{\bf A principal value calculation.}
We follow the approach of \cite{ahm1}, and calculate a certain principal value
integral.
 
\begin{lem}
Fix $z\in\D^e$. For any $l=0,1,2,\ldots$, we have 
\[ 
\pv 
\int_{\mathbb{C}}w^{-l} \berd^{\langle z\rangle}_{m,n,\Rpar}(w)\diff A(w)
\to z^{-l}, 
\]
as $m,n\to+\infty$ with $n=m+\Ordo(1)$.  
\end{lem}

\begin{proof}
The case $\Rpar=1$ was treated in \cite{ahm1}, so we may from now on assume 
that $\Rpar\ge2$. We write
\[K_{m,n,\Rpar}=K_{m,n,\Rpar}^I+K_{m,n,\Rpar}^{II},
\]
where
\[
K_{m,n,\Rpar}^I(z,w):=m\sum_{r=0}^{\Rpar-1}\sum_{i=0}^{n-r-1} 
\frac{r!}{(r+i)!}(mz \bar{w})^i L_r^{i}(m|z|^2)L_r^{i}(m|w|^2)
\]
and
\[
K_{m,n,\Rpar}^{II}(z,w):=
m\sum_{j=0}^{\Rpar-2} \sum_{k=0}^{\Rpar-j-1} \frac{j!}{(j+k)!} 
(m\bar{z}w)^k L^{k}_{j}(m|z|^2)L^{k}_{j}(m|w|^2). 
\]
It follows that the expression $|K_{m,n,\Rpar}|^2$ decomposes accordingly:
\begin{equation}
|K_{m,n,\Rpar}(z,w)|^2
=|K_{m,n,\Rpar}^I(z,w)|^2+|K_{m,n,\Rpar}^{II}(z,w)|^2
+2\re[ K_{m,n,\Rpar}^I(z,w)\overline{K_{m,n,\Rpar}^{II}(z,w)}].
\label{eq-decomp}
\end{equation}
We first consider the following integral involving $|K_{m,n,\Rpar}^I|^2$:
\begin{multline}
\pv \int_{\mathbb{C}} w^{-l} 
|K^I_{m,n,\Rpar}(z,w)|^2\e^{-m|w|^2} \diff A(w) 
\label{extpoint-rivi5} 
\\
=m^2 \sum_{r_1,r_2=0}^{\Rpar-1}\sum_{i_1=0}^{n-r_1-1} \sum_{i_2=0}^{n-r_2-1}
\pv \int_{\C} w^{-l} 
\frac{r_1!r_2!}{(r_1+i_1)!(r_2+i_2)!}(mz\bar{w})^{i_1}(m\bar{z}w)^{i_2}
\\
\times L_{r_1}^{i_1}(m|z|^2)L^{i_1}_{r_1}(m|w|^2)L^{i_2}_{r_2}(m|z|^2)
L^{i_2}_{r_2}(m|w|^2) \e^{-m|w|^2}\diff A(w)
\\
=m^2 z^{-l}
\sum_{r_1,r_2=0}^{\Rpar-1}\sum_{i_1=0}^{n-r_1-1}\sum_{i_2=0}^{n-r_2-1}
\frac{m^{i_1+i_2}r_1!r_2!}{(r_1+i_1)!(r_2+i_2)!}\delta_{i_2,i_1+l}\,|z|^{2i_2}
L_{r_1}^{i_1}(m|z|^2)L^{i_2}_{r_2}(m|z|^2)
\\
\times\int_{\C} |w|^{2i_1} L^{i_1}_{r_1}(m|w|^2)
L^{i_2}_{r_2}(m|w|^2) \e^{-m|w|^2}\diff A(w),
\end{multline}
where the delta is understood in Kronecker's sense.
The identity (for $p=1,2,3,\ldots$)
\begin{equation} 
\label{recurrencerelation}
L^{\alpha+p}_r(x)= \sum_{s=0}^r \binom{r-s+p-1}{p-1}\,
L_{s}^{\alpha}(x)
\end{equation}
plus the standard orthogonality properties of the Laguerre polynomials 
gives that
\[
\int_0^{+\infty}t^{i_1}L_{r_1}^{i_1}(t)L_{r_2}^{i_1+l}(t)\e^{-t}\diff t=
\frac{(r_1+i_1)!(r_2-r_1+l-1)!}{r_1!(r_2-r_1)!(l-1)!},
\]
where the right-hand side should be interpreted as $0$ for $r_2<r_1$. By
polar coordinates, then, we have
\[
\int_\C |w|^{2i_1}L_{r_1}^{i_1}(m|w|^2)L_{r_2}^{i_1+l}(m|w|^2)\e^{-m|w|^2}
\diff A(w)=
m^{-i_1-1}\frac{(r_1+i_1)!(r_2-r_1+l-1)!}{r_1!(r_2-r_1)!(l-1)!},
\]
and \eqref{extpoint-rivi5} simplifies to
\begin{multline}
\label{extpoint-rivi5.5}
\pv \int_{\mathbb{C}} w^{-l} 
|K^I_{m,n,\Rpar}(z,w)|^2\e^{-m|w|^2} \diff A(w) 
\\
=m\,z^{-l}
\sum_{r_1=0}^{\Rpar-1}\sum_{r_2=r_1}^{\Rpar-1}\sum_{i_1=0}^{n-l-r_2-1}
\frac{r_2!(r_2-r_1+l-1)!}{(r_2-r_1)!(l-1)!(r_2+i_1+l)!}
\,(m|z|^{2})^{i_1+l}L_{r_1}^{i_1}(m|z|^2)L^{i_1+l}_{r_2}(m|z|^2),
\end{multline}
provided $n$ is so big that $\Rpar+l\le n$. Next, we apply Lemma 
\ref{laguerreestimaatti} and \eqref{eq-triv1} (using that $z\in\D^e$) 
to arrive at
\begin{multline}
\label{extpoint-rivi5.6}
\sum_{i_1=0}^{n-l-r_2-1}\frac{r_1!r_2!}{(r_2+i_1+l)!} 
\,(m|z|^{2})^{i_1+l}\big|L_{r_1}^{i_1}(m|z|^2)L^{i_1+l}_{r_2}(m|z|^2)\big|
\\
\le\sum_{i_1=0}^{n-l-r_2-1}\frac{1}{(r_2+i_1+l)!} 
\,(m|z|^{2})^{i_1+l}(m|z|^2+i_1+r_1)^{r_1}(m|z|^2+i_1+l+r_2)^{r_2}
\\
\le 3^{r_1+r_2}\sum_{i_1=0}^{n-l-r_2-1}\frac{1}{(r_2+i_1+l)!} 
\,(m|z|^{2})^{r_1+r_2+i_1+l}\le 3^{r_1+r_2}(m|z|^2)^{r_1}E_{n-1}(m|z|^2).
\end{multline}
On the other hand, by the estimate from below in Lemma 
\ref{laguerreestimaatti},
\begin{multline} 
\label{berezin-outside-onepointestimate}
K_{m,n,\Rpar}(z,z) \geq K^{I}_{m,n,\Rpar}(z,z)=
m\sum_{r=0}^{\Rpar-1} \sum_{i=0}^{n-r-1} \frac{r!}{(r+i)!} 
(m|z|^2)^i L^{i}_{r}(m|z|^2)^2
\\
\ge m\sum_{r=0}^{\Rpar-1} \sum_{i=0}^{n-r-1} \frac{1}{r!(r+i)!} 
(m|z|^2)^i [m|z|^2-\beta(n)]^{2r}
\end{multline}
provided that $m|z|^2\ge\beta(n)$, where 
\[
\beta(n):=n+q-4+2\sqrt{\tfrac14+(q-2)(n-1)}=n+\Ordo(\sqrt{n}).
\]
As we assume $z\in\D^e$ and $n=m+\Ordo(1)$, we must have
\[
m|z|^2-\beta(n)=m(|z|^2-1)+m-\beta(n)\ge \tfrac12 m(|z|^2-1)
\]
for big enough $m,n$, and so, by \eqref{berezin-outside-onepointestimate},
\begin{multline} 
\label{berezin-outside-onepointestimate2}
K_{m,n,\Rpar}(z,z) \geq K^{I}_{m,n,\Rpar}(z,z)
\ge m\sum_{r=0}^{\Rpar-1} \sum_{i=0}^{n-r-1} \frac{4^{-r}}{r!(r+i)!} 
(m|z|^2)^{i+2r} 
\\
\ge\frac{4^{1-q}m}{(q-1)!}(m|z|)^{q-1}[E_{n-1}(m|z|^2)-
E_{q-2}(m|z|^2)]\ge\frac{4^{-q}m}{(q-1)!}(m|z|)^{q-1}E_{n-1}(m|z|^2).
\end{multline}
Here, we used that the $E_{n-1}(m|z|^2)$ is much bigger than $E_{q-2}(m|z|^2)$ 
as $m,n$ both grow. Let us look at the contribution from $0\le r_1\le q-2$
in the right-hand side of \eqref{extpoint-rivi5.5} using the estimate
\eqref{extpoint-rivi5.6}:
\begin{multline*}
m\,|z|^{-l}
\sum_{r_1=0}^{\Rpar-2}\sum_{r_2=r_1}^{\Rpar-1}\sum_{i_1=0}^{n-l-r_2-1}
\frac{r_2!(r_2-r_1+l-1)!}{(r_2-r_1)!(l-1)!(r_2+i_1+l)!}
\,(m|z|^{2})^{i_1+l}\big|L_{r_1}^{i_1}(m|z|^2)L^{i_1+l}_{r_2}(m|z|^2)\big|
\\
\le 
m\,\frac{|z|^{-l}}{(l-1)!}\sum_{r_1=0}^{\Rpar-2}\sum_{r_2=r_1}^{\Rpar-1}
\frac{(r_2-r_1+l-1)!}{r_1!(r_2-r_1)!}3^{r_1+r_2}(m|z|^2)^{r_1}E_{n-1}(m|z|^2)
\\
\le m C_1(q,l)(m|z|^2)^{q-2}E_{n-1}(m|z|^2),
\end{multline*}
for an appropriate positive constant $C_1(q,l)$. As we combine this estimate
with \eqref{berezin-outside-onepointestimate2}, we obtain
\begin{multline*}
m\,\frac{|z|^{-l}}{K_{m,n,q}(z,z)}
\sum_{r_1=0}^{\Rpar-2}\sum_{r_2=r_1}^{\Rpar-1}\sum_{i_1=0}^{n-l-r_2-1}
\frac{r_2!(r_2-r_1+l-1)!}{(r_2-r_1)!(l-1)!(r_2+i_1+l)!}
\\
\times(m|z|^{2})^{i_1+l}\big|L_{r_1}^{i_1}(m|z|^2)L^{i_1+l}_{r_2}(m|z|^2)\big|
\le C_2(q,l)m^{-1}=\Ordo(m^{-1})\longrightarrow0,
\end{multline*}
as $m,n\to+\infty$ in the prescribed fashion. So the contribution to
\eqref{extpoint-rivi5.5} which comes from $0\le r_1\le q-2$ is negligible
from the point of view of the Berezin density. It remains to consider the 
contribution from $r_1=q-1$. The corresponding part of the sum in the
right-hand side of \eqref{extpoint-rivi5.5} equals
\begin{equation*}
m\,z^{-l}
\sum_{i_1=0}^{n-l-q}
\frac{(q-1)!}{(q+i_1+l-1)!}
\,(m|z|^{2})^{i_1+l}L_{q-1}^{i_1}(m|z|^2)L^{i_1+l}_{q-1}(m|z|^2),
\end{equation*}
and we now claim that
\begin{equation}
\label{extpoint-rivi5.7}
m\,\frac{z^{-l}}{K_{m,n,q}(z,z)}
\sum_{i_1=0}^{n-l-q}
\frac{(q-1)!}{(q+i_1+l-1)!}
\,(m|z|^{2})^{i_1+l}L_{q-1}^{i_1}(m|z|^2)L^{i_1+l}_{q-1}(m|z|^2)
\longrightarrow z^{-l},
\end{equation}
as $m,n\to+\infty$ in the given fashion.
We use the recurrence relation (\ref{recurrencerelation}) to write 
\begin{equation} 
\label{eq-triv2}
L^{i_1}_{\Rpar-1}(x) = L^{i_1+l}_{\Rpar-1}(x)- \sum_{s=0}^{\Rpar-2} 
\binom{q+l-s-2}{l-1}\, L^{i_1}_s(x).
\end{equation}
A straightforward argument allows us to show that as insert this into 
\eqref{extpoint-rivi5.7}, the sum that is subtracted on the right-hand
side of \eqref{eq-triv2} makes asymptotically no contribution to the sum
in \eqref{extpoint-rivi5.7}. Consequently, \eqref{extpoint-rivi5.7} is 
equivalent to having
\begin{align}
\label{eq-triv3}
m\,\frac{z^{-l}}{K_{m,n,\Rpar}(z,z)}\sum_{i_1=0}^{n-l-q} \frac{(\Rpar-1)!}
{(\Rpar+i_1+l-1)!} (m|z|^2)^{i_1+l} L^{i_1+l}_{\Rpar-1}(m|z|^2)^2
\longrightarrow z^{-l}  
\end{align}
as $m,n\to+\infty$ in the given fashion.
If we insert the expression \eqref{eq-Kmnq} defining $K_{m,n,q}(z,z)$, we 
see that (there is some cancellation of terms)
\begin{multline}
z^{-l}-m\,\frac{z^{-l}}{K_{m,n,\Rpar}(z,z)}\sum_{i_1=0}^{n-l-\Rpar} 
\frac{(\Rpar-1)!}
{(\Rpar+i_1+l-1)!} (m|z|^2)^{i_1+l} L^{i_1+l}_{\Rpar-1}(m|z|^2)
L^{i_1+l}_{\Rpar-1}(m|z|^2)  
\\
=\frac{z^{-l}}{K_{m,n,\Rpar}(z,z)}\Bigg\{
m\sum_{r=0}^{\Rpar-2}\sum_{i=0}^{n-r-1} \frac{r!}{(r+i)!} 
(m|z|^2)^i L_r^{i}(m|z|^2)^2 
\\
+m \sum_{j=1}^{\Rpar-2}\sum_{k=1}^{\Rpar-j-1} \frac{j!}{(j+k)!} 
(m|z|^2)^k L^{k}_{j}(m|z|^2)^2 
+m\sum_{i_1=0}^{l-1} \frac{(\Rpar-1)!}
{(\Rpar+i-1)!} (m|z|^2)^{i} L^{i}_{\Rpar-1}(m|z|^2)^2
\Bigg\}.
\end{multline}
By careful application of Lemma \ref{laguerreestimaatti} to all the
Laguerre polynomials in this expression, while inserting 
\eqref{berezin-outside-onepointestimate} to control the denominator, 
we indeed get \eqref{eq-triv3}. So, after a lot of effort, we have obtained
that
\begin{equation}
\label{eq-triv5}
\pv 
\int_{\mathbb{C}} w^{-l} 
\frac{|K^I_{m,n,\Rpar}(z,w)|^2}{K_{m,n,q}(z,z)}\e^{-m|w|^2} \diff A(w) 
\longrightarrow z^{-l}
\end{equation}
as $m,n\to+\infty$ with $n=m+\Ordo(1)$. 
Analogous but slightly easier arguments (left to the interested reader) 
show that 
\begin{equation*}
\pv 
\int_{\mathbb{C}}w^{-j}\frac{|K^{II}_{m,n,\Rpar}(z,w)|^2}{K_{m,n,q}(z,z)} 
\e^{-|w|^2}\diff A(w) \longrightarrow 0 
\end{equation*}
and
\[
\int_{\mathbb{C}}w^{-j}\re\Bigg\{\frac{K^{I}_{m,n,\Rpar}(z,w)
\overline{K^{II}_{m,n,\Rpar}(z,w)}}{K_{m,n,q}(z,z)}\Bigg\}
\,\e^{-|w|^2}\diff A(w) \longrightarrow 0,
\]
again as $m,n\to+\infty$ with $n=m+\Ordo(1)$.
Finally, we put everything together based on the decomposition 
\eqref{eq-decomp}:
\[
\pv\int_{\mathbb{C}}w^{-j}\berd^{\langle z\rangle}_{m,n,q}(w)=\pv\int_{\C}
\frac{|K_{m,n,\Rpar}(z,w)|^2}{K_{m,n,q}(z,z)}\,\e^{-|w|^2}\diff A(w) 
\longrightarrow z^{-l},
\]
as $m,n\to+\infty$ with $n=m+\Ordo(1)$. This ends the proof. 
\end{proof}
\medskip

\noindent{\bf Convergence to harmonic measure.}
We now show that even in the principal value sense, the Berezin density 
tends to avoid the interior of the unit disk.

\begin{lem}
Fix a real parameter $\rho$ with $0<\rho<1$ and a point $z\in\D^e$.
Then, for $l=0,1,2,\ldots$, we have the convergence 
\[ 
\pv \int_{\D(0,\rho)}w^{-l}\berd^{\langle z\rangle}_{m,n,\Rpar}(w)\diff A(w) 
\longrightarrow0, 
\]
as $m,n\to+\infty$ with $n=m+\Ordo(1)$. 
\end{lem}

\begin{proof}
In terms of the decomposition \eqref{eq-decomp}, we will focus on the 
the term $|K^I_{m,n,q}|^2$ and leave the other two
to the reader (the necessary arguments are similar but slightly easier). 
The analogue of \eqref{extpoint-rivi5} reads
\begin{multline}
\label{eq-triv7}
\pv \int_{\D(0,\rho)} w^{-l} 
|K^I_{m,n,q}(z,w)|^2\e^{-m|w|^2} \diff A(w)
\\
= m^{2} z^{-l} \sum_{r_1,r_2=0}^{\Rpar-1}
\sum_{i_1=0}^{\min\{n-r_1-1,n-l-r_2-1\}}
\frac{r_1!r_2!}{(r_1+i_1)!(r_2+i_1+l)!}
m(|z|^2)^{i_1+l}L^{i_1}_{r_1}(m|z|^2)L^{i_1+l}_{r_2}(m|z|^2) 
\\
\times\int_{\D(0,\rho)} (m|w|^2)^{i_1}L^{i_1}_{r_1}(m|w|^2)
L^{i_1+l}_{r_2}(m|w|^2)\e^{-m|w|^2} \diff A(w). 
\end{multline}
By Lemma \ref{laguerreestimaatti}, we have
\begin{multline*}
\int_{\D(0,\rho)} (m|w|^2)^{i_1}\big|L^{i_1}_{r_1}(m|w|^2)
L^{i_1+l}_{r_2}(m|w|^2)\big|\e^{-m|w|^2} \diff A(w)
\\
\le\frac{1}{r_1!r_2!}\int_{\D(0,\rho)} (m|w|^2)^{i_1}(m|w|^2+i_1+r_1)^{r_1}
(m|w|^2+i_1+l+r_2)^{r_2}\e^{-m|w|^2} \diff A(w)
\\
\le
\frac{2^{r_1+r_2-2}}{r_1!r_2!}\int_{\D(0,\rho)} (m|w|^2)^{i_1}
[(m|w|^2)^{r_1}+(i_1+r_1)^{r_1}][
(m|w|^2)^{r_2}+(i_1+l+r_2)^{r_2}]\e^{-m|w|^2} \diff A(w)
\\
=\frac{2^{r_1+r_2-2}}{r_1!r_2!m}\Bigg\{
\int_0^{m\rho^2}t^{i_1+r_1+r_2}\e^{-t}\diff t+
(i_1+r_1)^{r_1}\int_0^{m\rho^2}t^{i_1+r_2}\e^{-t}\diff t
\\
+(i_1+l+r_2)^{r_2}\int_0^{m\rho^2}t^{i_1+r_1}\e^{-t}\diff t
+(i_1+r_1)^{r_1}(i_1+l+r_2)^{r_2}\int_0^{m\rho^2}t^{i_1}\e^{-t}\diff t\Bigg\}.
\end{multline*}
In terms of the function
\[
\chi(a,b):=\frac{1}{\Gamma(a+1)}\int_0^b t^a \e^{-t}\diff t,
\qquad a,b\in[0,+\infty[,
\]
which takes values in the interval $[0,1[$, the estimate becomes
\begin{multline*}
\int_{\D(0,\rho)} (m|w|^2)^{i_1}\big|L^{i_1}_{r_1}(m|w|^2)
L^{i_1+l}_{r_2}(m|w|^2)\big|\e^{-m|w|^2} \diff A(w)
\\
\le\frac{2^{r_1+r_2-2}}{r_1!r_2!m}\Bigg\{
(i_1+r_1+r_2)!\chi(i_1+r_1+r_2,m\rho^2)+
(i_1+r_1)^{r_1}(i_1+r_2)!\chi(i_1+r_2,m\rho^2)
\\
+(i_1+l+r_2)^{r_2}(i_1+r_1)!\chi(i_1+r_1,m\rho^2)
+(i_1+r_1)^{r_1}(i_1+l+r_2)^{r_2}i_1!\chi(i_1,m\rho^2)\Bigg\}.
\end{multline*}
and if we use that $a\mapsto\chi(a,b)$ is decreasing for fixed $b$
(a direct calculation involving derivatives suffices to verify this),
we get
\begin{multline}
\label{eq-triv8}
\int_{\D(0,\rho)} (m|w|^2)^{i_1}\big|L^{i_1}_{r_1}(m|w|^2)
L^{i_1+l}_{r_2}(m|w|^2)\big|\e^{-m|w|^2} \diff A(w)
\\
\le\frac{2^{r_1+r_2-2}}{r_1!r_2!m}\Bigg\{
(i_1+r_1+r_2)!+
(i_1+r_1)^{r_1}(i_1+r_2)!
\\
+(i_1+l+r_2)^{r_2}(i_1+r_1)!
+(i_1+r_1)^{r_1}(i_1+l+r_2)^{r_2}i_1!\Bigg\}\chi(i_1,m\rho^2).
\end{multline}
Next, since $z\in\D^e$, $n=m+\Ordo(1)$, and $i_1+l\le n-1$, we may use 
another aspect of Lemma \ref{laguerreestimaatti} to see that
\begin{equation}
\label{eq-triv9}
\big|L^{i_1}_{r_1}(m|z|^2)L^{i_1+l}_{r_2}(m|z|^2)\big|\le\frac{1}{r_1!r_2!}
(m|z|^2)^{r_1+r_2},
\end{equation}
provided $m,n$ are big enough. As we combine the equality \eqref{eq-triv7} 
with the estimates \eqref{eq-triv8} and \eqref{eq-triv9}, and use
some well-understood comparisons of factorials and powers, we arrive at
\begin{multline}
\label{eq-triv10}
\bigg|\pv \int_{\D(0,\rho)} w^{-l} 
|K^I_{m,n,q}(z,w)|^2\e^{-m|w|^2} \diff A(w)\bigg|
\\
\le m\,C_3(q,l)|z|^{-l+2q-2}\sum_{i_1=0}^{n-l-1}
\frac{(m|z|^2)^{i_1+l}}{(i_1+l)!}\chi(i_1,m\rho^2),
\end{multline}
for some appropriate positive constant $C_3(q,l)$. The function 
$\chi(i_1,m\rho^2)$ drops off exponentially quickly to $0$ as $i_1$ 
exceeds $m\rho^2$ by a margin greater than $\Ordo(m^{1/2})$, as can be seen, 
e.g., by an application of the Central Limit Theorem (compare with the
next section). This means that 
effectively we are summing up to $m|\rho|^2+\Ordo(m^{1/2})$ in the right-hand 
side expression of \eqref{eq-triv10}, which does not permit the sum to compare 
with the size of $K_{m,n,q}(z,z)$; cf. 
\eqref{berezin-outside-onepointestimate2}.
This results in the convergence
\begin{equation*}
\pv \int_{\D(0,\rho)} w^{-l} 
\frac{|K^I_{m,n,q}(z,w)|^2}{K_{m,n,q}(z,z)}\,\e^{-m|w|^2} \diff A(w)
\longrightarrow0
\end{equation*}
as $m,n\to+\infty$ with $n=m+\Ordo(1)$. Together with the estimates which were
left as an exercise to the reader, we get
\begin{equation*}
\pv \int_{\D(0,\rho)} w^{-l} 
\frac{|K_{m,n,q}(z,w)|^2}{K_{m,n,q}(z,z)}\,\e^{-m|w|^2} \diff A(w)
\longrightarrow0
\end{equation*} 
as $m,n\to+\infty$ with $n=m+\Ordo(1)$, which amounts to the assertion of the
lemma.
\end{proof}

As in the proof of the Theorem 2.11 in \cite{ahm1}, we may now conclude
the following. 
  
\begin{thm}
Fix $z\in\D^e$ and a bounded continuous function $g$ on $\mathbb{C}$. Then 
\[ 
\int_{\C}g(w)\berd^{\langle z\rangle}_{m,n,\Rpar}(w)\diff A(w)\to\int_{\C}g(w)
\diff\omega(w,z, \mathbb{D}^*) , 
\] 
as $m,n \to+\infty$ with $n=m+\Ordo(1)$. Here, 
$ \diff\omega(w,z, \mathbb{D}^*) $ is harmonic measure with respect to the 
point $z$ and the domain $\D^e$.
\end{thm}

\section{Poly-Bargmann transforms}

\noindent{\bf Purpose of the section.}
In this section, we discuss the poly-Bargmann transforms, a generalization of
the classical Bargmann transform, which are needed later when we analyze 
the Berezin density at a boundary point. The poly-Bargmann transforms appeared
in Vasilevski's paper \cite{vas}, where the basic properties were presented. 
\medskip

\noindent{\bf The Hermite polynomials and the Bargmann transform.} 
We denote by $H_j$ the $j$-th Hermite polynomial with respect to the Gaussian 
weight $\e^{-\frac12t^2}$ (``probabilistic Hermite polynomials''). The 
generating function identity
\[ 
\e^{tz-\frac12z^2} = \sum_{j=0}^{+\infty} H_j(t)\frac{z^j}{j!}
\]
allows us to write
\[
\frac{1}{(2\pi)^{1/4}}\e^{zt-\frac12z^2-\frac14t^2}=  
\sum_{j=0}^{+\infty} \frac{H_j(t)\e^{-t^2/4}}{(2\pi)^{1/4}\sqrt{j!}}
\times\frac{z^j}{\sqrt{j!}}.
\]
We recall the standard definition of the Bargmann transform:
\[
\Bop[f](z)= \frac{1}{(2 \pi)^{1/4}} \int_{\R}
 \e^{zt-\frac12z^2-\frac14t^2}f(t)\,\diff t, \qquad
f \in L^2(\R). 
\]
As the function systems 
\[
\Big\{\frac{1}{(2\pi)^{1/4}\sqrt{j!}}H_j(t)\e^{-\frac14t^2} 
\Big\}_{j=0}^{+\infty}\quad\text{and}\quad \Big\{\frac{z^j}{\sqrt{j!}}
\Big\}_{j=0}^{+\infty}
\] 
form orthonormal bases for $L^2(\mathbb{R})$ and the Bargmann-Fock space 
$A^2_{1,1}(\C)$ (this is $A^2_{m,q}(\C)$ with $m=q=1$), respectively,  
we obtain the following well-known fact.
 
\begin{prop}
The Bargmann transform $\Bop:L^2(\R)\to A_{1,1}^2(\C)$ acts isometrically and
bijectively, and for each $j=0,1,2,\ldots$, the basis function
$(2\pi)^{-1/4}(j!)^{-1/2}H_j(t)\e^{-\frac14t^2}$ is mapped to the basis 
function $(j!)^{-1/2}z^j$. 
\end{prop}
\medskip

\noindent{\bf A class of auxiliary operators.}
Let $\partial_z,\bar\partial_z$ denote the standard Wirtinger differential 
operators
\[
\partial_z:=\tfrac12(\partial_x-\imag\partial_y),\,\,\,
\bar\partial_z:=\tfrac12(\partial_x+\imag\partial_y),
\quad\text{where}\,\,\, z=x+\imag y.
\]
For $r=0,1,2,\ldots$, we introduce the operator 
\begin{equation*}
\Tope_{r}[f](z):=\frac{1}{\sqrt{r!}} \e^{|z|^2}\,
\partial_z^r\big\{f(z)\e^{-|z|^2}\big\}, 
\end{equation*}
with the semi-group property $\Tope_1\circ\Tope_{r-1}=r^{1/2}\Tope_r$.
We also consider the dilated variant
\begin{equation*}
\Tope_{m,r}[f](z) = \Tope_r[f_{m^{-1/2}}](m^{1/2}z), 
\end{equation*}
where $f_{m^{-1/2}}(z)=f(m^{-1/2}z)$. It has the semi-group property 
$\Tope_{m,1}\circ\Tope_{m,r-1}=r^{1/2}\Tope_{m,r}$, and may be expressed
in more concrete terms:
\begin{equation*}
\Tope_{m,r}[f](z) = \frac{m^{-r/2}}{\sqrt{r!}} \e^{m|z|^2}\,
\partial_z^r\big\{f(z)\e^{-m|z|^2}\big\}.
\end{equation*}  
We now study the effect of $\Tope_r$ on the basis elements $(j!)^{-1/2}z^j$.

\begin{prop} 
\label{T_r-prop}
For $j\ge r$, we have
\[
\Tope_r\bigg[\frac{z^j}{\sqrt{j!}}\bigg]=
\frac{1}{\sqrt{r!}}\e^{|z|^2}
\partial^r_z\bigg\{\e^{-|z|^2}\frac{z^j}{\sqrt{j!}}\bigg\}= 
\sqrt{\frac{r!}{j!}}z^{j-r}L^{j-r}_r(|z|^2)
\]
while for $j \leq r$, 
\[
\Tope_r\bigg[\frac{z^j}{\sqrt{j!}}\bigg]=
\frac{1}{\sqrt{r!}}\e^{|z|^2}\partial^r_z\bigg\{\e^{-|z|^2}
\frac{z^j}{\sqrt{j!}}\bigg\}= 
(-1)^{r-j}\sqrt{\frac{n!}{r!}}\,\bar{z}^{r-j}L^{r-j}_j(|z|^2)
\]
\end{prop}

\begin{proof}
The proof is based on an induction argument. 
The statement is obviously true for $r=0$ and all $j$. 
So, by induction, we assume that the statement holds for some 
$r-1\geq 0$ and all $j$. In case $j\ge r$, we then have
\begin{multline*}
\partial^r_z\{\e^{-|z|^2}z^j\}= 
\partial_z\big\{(r-1)!\,z^{j-r+1}L^{j-r+1}_{r-1}(|z|^2)\e^{-|z|^2} 
\big\}
\\
=(r-1)! z^{j-r}
\big\{(j-r+1)L^{j-r+1}_{r-1}(|z|^2)-L^{j-r+2}_{r-2}(|z|^2)|z|^2  
-|z|^2L^{j-r+1}_{r-1}(|z|^2)\big\}\e^{-|z|^2} 
\\
= r! L^{j-r}_r(|z|^2)z^{j-r}\e^{-|z|^2},
\end{multline*}
if we use the standard identity $rL^{\alpha}_r(x) = (\alpha + 1 -x) 
L^{\alpha+1}_{r-1}(x) - xL^{\alpha+2}_{r-2}(x)$.
Next, in case $j\leq r-1$, we have instead
\begin{multline*}
\partial^r_z\{\e^{-|z|^2}z^j\}= (-1)^{r-1-j} \partial 
\big\{
j!\,\bar{z}^{r-1-j}L^{r-1-j}_j(|z|^2)\e^{-|z|^2} 
\big\} 
\\
=(-1)^{r-j} j! \bar{z}^{r-j} \big\{L^{r-j}_{j-1}(|z|^2)+ 
L^{r-j-1}_{j}(|z|^2) \big\}\e^{-|z|^2}
=(-1)^{r-j} j!\,\bar{z}^{r-j}L^{r-j}_j (|z|^2)\e^{-|z|^2},
\end{multline*}
which completes the proof.
\end{proof}
\medskip

\noindent{\bf Pure poly-analytic Fock spaces.} 
Write $e_j(z):=(j!)^{-1/2}z^j$ for $j=0,1,2,\ldots$, which 
functions form the standard orthonormal basis for the space $A^2_{1,1}(\C)$.
The function $\Tope_r[e_j]\in A^2_{1,r+1}(\C)$ is then a polynomial in 
$z,\bar z$, where the degree in $z$ remains equal to $j$, and the degree in 
$\bar z$ equals $r$. For general positive $m$, the functions $e_{j,m}(z):=
(j!)^{-1/2}m^{\frac12(j+1)}z^j$, with $j=0,1,2,\ldots$, form an orthonormal
basis for the space $A^2_{m,1}(\C)$. The functions $\Tope_{m,r}[e_{j,m}]$
are computed using Proposition \ref{T_r-prop} above, and we then recognize
that we can identify them with the basis elements which appear in
Proposition \ref{on-kanta}. We clearly have that
\[
\mathrm{span}\big\{\Tope_{m,r}[e_{j,m}]:\,\, 0\le j\le n-1,\,\,0\le r\le q-1
\big\}=\mathrm{Pol}_{m,n,q},
\]
and in view of the orthogonality properties in Proposition \ref{on-kanta}, 
we also must have
\begin{equation}
\mathrm{span}\{\Tope_{m,r}[e_{j,m}]:\,\, 0\le j\le n-1\}
=\mathrm{Pol}_{m,n,r+1}\ominus \mathrm{Pol}_{m,n,r},
\label{eq-resspace1}
\end{equation}
where the ``$\ominus$'' is with respect to the inner product in 
$L^2(\C,\e^{-m|z|^2})$.
Following Vasilevski \cite{vas}, then, we define the pure polyanalytic Fock 
space of level $r+1$:
\begin{equation}
\Fspace_{m,n,r+1}:=\mathrm{Pol}_{m,n,r+1}\ominus \mathrm{Pol}_{m,n,r},\quad
\delta A^2_{m,n,r+1}(\C):=A^2_{m,r+1}(\C)\ominus A^2_{m,r}(\C),
\label{eq-resspace2}
\end{equation} 
with the understanding that $\mathrm{Pol}_{m,n,0}:=\{0\}$ and 
$A^2_{m,0}(\C):=\{0\}$. 
The operator $\Tope_{m,r}$ now becomes an isometric isomorphism 
\[
\Tope_{m,r}:A^2_{m,1}(\C)\to\delta A^2_{m,n,r+1}(\C),\quad
\Tope_{m,r}:\mathrm{Pol}_{m,n,1}\to\Fspace_{m,n,r+1}.
\]
We obtain the orthogonal decompositions
\begin{align*} 
\mathrm{Pol}_{m,n,\Rpar} &= \bigoplus_{r=0}^{\Rpar-1} \Fspace_{m,n,r+1},
& A^2_{m,\Rpar}(\C)=\bigoplus_{r=0}^{\Rpar-1} \delta A^2_{m,r+1}(\C).
\end{align*}
As a consequence, if $K_{\delta:m,n,r+1}$ and $K_{\delta:m,r+1}$ 
denote the reproducing kernels for the spaces $\Fspace_{m,n,r+1}$ and
$\delta A^2_{m,r+1}(\C)$, respectively, we must have that
\begin{align} 
\label{kernelsubspacedecomposition}
K_{m,n,\Rpar} &= \sum_{r=0}^{\Rpar-1} K_{\delta:{m,n,r+1}},
&  K_{m,\Rpar} &=\sum_{r=0}^{\Rpar-1} K_{\delta:{m,r}}.
\end{align}
\medskip

\noindent{\bf The poly-Bargmann transforms.} For $m=1$, the poly-Bargmann 
transform of level $r$ is defined by
\begin{align*}
\Bop_r[f]:= \Tope_r\circ\Bop[f].
\end{align*}

\begin{prop}
\label{prop-Br}
We have
\[ 
\Bop_r[f](z)= 
\frac{1}{(2\pi)^{1/4}\sqrt{r!}} \int_{\R}f(t)H_r(t-z-\bar{z})
\,\e^{tz-\frac12z^2-\frac14t^2} \diff t,\qquad r=0,1,2,\ldots.
\]
\end{prop}
\begin{proof}
We proceed by induction. For $r=0$, the formula reduces to the usual Bargmann
transform (if we recall that $H_0=1$). By the induction hypothesis, we suppose 
therefore that the formula is valid for all integers up to $r-1$. 
From the semi-group property $\Tope_1\circ\Tope_{r-1}=r^{1/2}\Tope_r$, 
we see that   
\[ 
\Bop_r[f]=\Tope_r\circ\Bop[f]=
r^{-1/2}\Tope_1\circ\Tope_{r-1}\circ\Bop[f](z)= 
r^{-1/2}\Tope_1\circ\Bop_{r-1}[f].
\]
Since the formula holds for $r-1$, the following calculation shows that it
holds for $r$ as well:
\begin{multline*}
r^{-1/2}\e^{|z|^2} 
\partial_z 
\bigg([(r-1)!]^{-1/2}H_{r-1}(t-z- \bar{z})
\,\e^{tz-\frac12z^2-\frac14t^2}\e^{-|z|^2}\bigg)  
\\
= (r!)^{-1/2}\left(-H'_{r-1}(t-z-\bar{z})+ 
(t-z- \overline{z})H_{r-1}(t-z-\bar{z}) \right)
\,\e^{tz-\frac12z^2-\frac14t^2}
\\
= (r!)^{-1/2} H_r(t-z-\bar{z})\e^{tz-\frac12z^2-\frac14t^2}; 
\end{multline*}
here, we used the standard identity $H_{r}(x)=xH_{r-1}(x)-H'_{r-1}(x)$. 
The proof is complete. 
\end{proof}

\section{Reproducing kernel and Berezin density asymptotics for a boundary 
point}

\noindent{\bf Purpose of the section.}
In this section, we will calculate the limit of the blow-up berezin transform 
$\hat{\berd}^{\langle z\rangle}_{m,n,\Rpar}$ at a boundary point $z$, 
that is, $|z|=1$. 
There is no loss of generality to take $z=1$. 
Our strategy is to investigate the blow-up of the reproducing kernel of 
the space of analytic polynomials $\mathrm{Pol}_{m,n,1}$ (with $q=1$) first, 
and then use this information together with poly-Bargmann transform to lift the
the asymptotics to the context of the general polyanalytic spaces 
$\mathrm{Pol}_{m,n,\Rpar}$. 
\medskip

\noindent{\bf The central limit theorem revisited.}
The following improvement of the central 
limit theorem will be needed. Let 
$\mathrm{cdf}_X$ denote the cumulative distribution function 
of a real-valued random variable $X$ and let 
\[ 
\mathrm{erf}(z) = \frac{1}{\sqrt{2 \pi}} 
\int_{- \infty}^z \e^{-\frac12t^2} \diff t 
\]
be the error function. We shall write i.i.d. as shorthand for {\em 
independent identically distributed} in the context of random variables.
The following result is from \cite{berry}, \cite{esseen}.

\begin{thm}[Berry-Ess\'een]
Let $X_1, X_2, ...$ be i.i.d. real-valued random variables 
with $E(X_j)=0$, $E(X_j^2)=1$ and $E(|X_j|^3)=\rho<+\infty$
for all $j$. Also, let $Y_n = n^{-1/2}\sum_{j=1}^{n} X_j$.
Then there exists an absolute constant $C$ such that 
\[ 
\left| \mathrm{cdf}_{Y_n}(x) - \mathrm{erf}(x) \right| 
\leq \frac{C \rho}{\sqrt{n}}, \qquad x\in\R.
\]
\end{thm}
 
The Berry-Ess\'een theorem gives the following asymptotics for the partial
Taylor sums of the exponential function. 
 
\begin{lem} 
\label{erflemma}
We have
\[ 
\frac{E_{n-1}(x)}{\e^x} = \mathrm{erf}\bigg(\frac{n-x}{\sqrt{n}}\bigg) 
+ \Ordo(n^{-1/2}), 
\]
as $n\to+\infty$, uniformly in $x\in[0,+\infty[$.
\end{lem}

\begin{proof}
Let $X_1,...,X_n$ be independent exponentially distributed random variables 
on $[0,+\infty[$ with density $\e^{-x}$. It is well known that the 
sum $\sum_{j=1}^nX_j$ obeys a gamma distribution with the 
cumulative distribution function $1- \e^{-x}E_{n-1}(x)$. 
The random variables $X_1-1, ..., X_n-1$ all have zero mean and 
variance $1$, and the third moment is finite, so by the Berry-Ess\'een 
theorem, we have  
\begin{equation}
1-\e^{-x}E_{n-1}(x)=\text{cdf}_{\sum_{j=1}^n X_j}(x) 
=\mathrm{cdf}_{\frac{\sum_{j=1}^n X_j -n }{\sqrt{n}}}
\bigg(\frac{x-n}{\sqrt{n}}\bigg) 
=\mathrm{erf}\bigg(\frac{x-n}{\sqrt{n}}\bigg)+\Ordo(n^{-1/2}), 
\end{equation}
where the ``$\Ordo$'' term is uniform in $x$ as $n\to+\infty$. 
\end{proof}

This allows us to blow up the function $E_{n-1}(mz\bar w)$ when $z$ and $w$
are close to the point $1$ and $m,n\to+\infty$ with $n=m+\Ordo(1)$.
 
\begin{lem} \label{analyticbdrblowup}
Fix a positive real $\varepsilon$. For complex $\xi,\eta\in\C$, we then have
\[ 
\frac{E_{n-1}\left(m(1+m^{-1/2}\xi)(1+ m^{-1/2}\bar{\eta})\right)}
{\e^{m+ \sqrt{m}(\xi+\bar\eta)}} = \e^{\xi\bar{\eta}}
\,\mathrm{erf}(-\xi- \bar\eta)+\mathrm{O}(m^{-\frac12+\varepsilon}), 
\]
as $m,n\to+\infty$ while $n=m+\Ordo(1)$. 
Here, the ``$\Ordo$'' expression on the right-hand side is uniform on compact 
subsets of $\C$. 
\end{lem}

\begin{proof}
If we put 
\[
\zeta:=m(1+m^{-1/2}\xi)(1+m^{-1/2}\bar\eta)=m+m^{1/2}
(\xi+\bar\eta)+\xi\bar\eta,
\]
then 
\[
\frac{n-\zeta}{\sqrt{n}}=\frac{n-m-m^{1/2}(\xi+\bar\eta)-\xi\bar\eta}{\sqrt{n}}
=-\xi-\bar\eta+\Ordo(m^{-1/2}),
\]
where the ``O'' term is uniform on compact subsets. If we use that $E_{n-1}$
has only nonnegative Taylor coefficients, we get from Lemma \ref{erflemma}
that
\begin{equation}
\label{eq-t1}
\frac{|E_{n-1}(\zeta)|}{\e^{|\zeta|}}\le \frac{E_{n-1}(|\zeta|)}{\e^{|\zeta|}}
=\mathrm{erf}\bigg(\frac{n-|\zeta|}{\sqrt{n}}\bigg) 
+\Ordo(n^{-1/2}),
\end{equation}
with a uniform ``O'' term. So, for {\em real} $\xi,\eta$, we may deduce from
\eqref{eq-t1} the assertion of the lemma with $\varepsilon=0$. For general 
complex $\xi,\eta$, we note that
\begin{multline*}
|\zeta|=m\,|(1+m^{-1/2}\xi)(1+m^{-1/2}\bar\eta)|
\\
=m+m^{1/2}\re[\xi+\eta]+
\re\xi\re\eta+\tfrac12(\im\xi)^2+\tfrac12(\im\eta)^2+\Ordo(m^{-1/8})
\end{multline*}
uniformly in the domain where $\max\{|\xi|,|\eta|\}\le m^{1/8}$. As the
right-hand side of \eqref{eq-t1} is $\le\frac32$ for big $n$, we see that
\begin{multline*} 
\Bigg|\frac{E_{n-1}\left(m(1+m^{-1/2}\xi)(1+ m^{-1/2}{\eta})\right)}
{\e^{m+ \sqrt{m}(\xi+\eta)+\xi\bar\eta}}\Bigg|
\le\frac32\e^{-(\im\xi)(\im\eta)+\tfrac12(\im\xi)^2+\tfrac12(\im\eta)^2+
\Ordo(m^{-1/8})}
\\
\le2\e^{\frac12[\im\xi-\im\eta]^2} 
\end{multline*}
holds in the domain where $\max\{|\xi|,|\eta|\}\le m^{1/8}$, provided $m$
is big enough.
We need to show that he difference
\[ 
F_{m,n}(\xi,\eta):=\frac{E_{n-1}\left(m(1+m^{-1/2}{\xi})
(1+m^{-1/2}\bar\eta)\right)}{\e^{m+ \sqrt{m}(\xi+\bar\eta)+\xi\bar\eta}} - 
\mathrm{erf}(-\xi-\bar\eta) 
\]
is of order $\Ordo(m^{-1/2+\varepsilon})$ uniformly as $\xi,\eta$ remain 
confined to some compact subset of $\C$. We know that
$F_{m,n}(\xi,\eta)=\Ordo(m^{-1/2})$ uniformly when $\xi,\eta\in\R$ with
confined to $\max\{|\xi|,|\eta|\}\le m^{1/8}$. In view of the calculation we 
just made, we also have a good uniform estimate of $F_{m,n}(\xi,\eta)$ when 
$\xi,\eta\in\C$ with $\max\{|\xi|,|\eta|\}\le m^{1/8}$. By employing a 
standard technique involving the subharmonicity of 
$\xi\mapsto\log|F_{m,n}(\xi,\eta)|$, and certain classical estimates 
of harmonic measure, we can show that 
$F_{m,n}(\xi,\eta)=\Ordo(m^{\epsilon-1/2})$ holds uniformly when 
$\xi,\eta\in\C$ belong to a compact subset of $\C$, and in addition, 
$\eta\in\R$. Here, $\epsilon$ is a positive number which we can get as small 
as we like. A similar argument with $\eta$ in place $\xi$ worsens the 
control to $F_{m,n}(\xi,\eta)=\Ordo(m^{2\epsilon-1/2})$, but now the control 
is uniform when both $\xi,\eta$ are both complex and confined to some compact
subset. The proof is complete with $\varepsilon=2\epsilon$. 
\end{proof}
\medskip

\noindent{\bf The reproducing kernel for a subspace of the Fock space.}
We shall identify both the right-hand and the left-hand side expressions 
appearing in Lemma \ref{analyticbdrblowup} with reproducing kernels of 
certain Hilbert spaces of entire functions. This will be the case $r=0$
of the proposition below.

Let us agree to identify 
\[ 
L^2(\R_-) =\{ f\in L^2(\mathbb{R}):\,\, f(x)=0 \,\,\text{ for }\,\, x>0\}.
\]

\begin{prop} 
\label{polybargmannkernel}
For $r=0,1,2,\ldots$, the function 
\[ 
(\xi,\eta)\mapsto\frac{\e^{\xi\bar\eta}}{r! \sqrt{2 \pi}}
\int_{-\infty}^{-\xi-\bar\eta} H_r(t+\xi-\eta)
H_r(t+\bar\eta-\bar\xi)\e^{-\frac12t^2} \diff t  
\]
 is the reproducing kernel for the Hilbert space $\Bop_r[L^2(\R_-)]\subset 
A^2_{1,r+1}(\C)$.
\end{prop}

\begin{proof}
Let $M(\xi,\eta)=M_{\eta}(\xi)$ be the reproducing kernel for 
$\Bop_r[ L^2(\R_-)]$. This kernel has $M_\eta\in\Bop_r[L^2(\R_-)]$ and
\begin{multline*}
\frac{1}{\sqrt{r!}(2\pi)^{1/4}} \int_{\R}f(t)H_r(t-\eta-\bar{\eta})
\e^{t\eta-\frac12\eta^2-\frac14t^2}\diff t
= \Bop_r[f](\eta) \\
= \langle \Bop_r[f], M_{\eta} \rangle_{A^2_{1,r+1}(\C)} = 
\langle f, \Bop_r^{-1}[M_{\eta}] \rangle_{L^2(\R)} 
\end{multline*}
for all $\eta\in\C$ and all $f\in L^2(\R_-)$, which allows us to conclude that 
\begin{equation*}
\Bop_r^{-1}[M_{\eta}](t)  
=(r!)^{-1/2}(2\pi)^{-1/4}1_{]-\infty,0]}(t)\, H_r(t-\eta-\bar{\eta})\,
\e^{t\bar\eta-\frac12\bar{\eta}^2-\frac14t^2}.
\end{equation*}
After applying the operator $\Bop_r$ to both sides, we see that (cf. 
Proposition \ref{prop-Br})
\begin{multline*}
 M_{\eta}(\xi) = \frac{1}{r!\sqrt{2\pi}} \int_{-\infty}^{0} 
H_r(t-\xi-\bar{\xi}) \e^{t \xi- \frac12\xi^2-\frac14 t^2} 
H_r(t-\eta-\bar{\eta}) \e^{t\bar{\eta}-\frac12\bar{\eta}^2-\frac14t^2} 
\diff t 
\\
=\frac{\e^{\bar{\eta} \xi}}{r! \sqrt{2\pi}} 
\int_{-\infty}^{0} H_r(t-\xi-\bar{\xi}) 
H_r(t-\eta-\bar{\eta}) \e^{-\frac{1}{2}(t - \xi - \bar{\eta})^2} 
\diff t 
\\
=\frac{\e^{\bar{\eta}\xi}}{r! \sqrt{2\pi}} \int_{-\infty}^{-\bar{\eta}-\xi} 
H_r(t + \bar{\eta} - \bar{\xi}) 
H_r(t+\xi-\eta) \e^{-\frac{1}{2}t^2} \diff t. 
\end{multline*}
The proof is complete.
\end{proof}

\begin{rem}
The special case $r=0$ of the kernel in Proposition \ref{polybargmannkernel} 
is $\e^{\xi\bar\eta}\mathrm{erf}(-\xi-\bar\eta)$, 
which is what we have on the right-hand side in Lemma \ref{analyticbdrblowup}. 
\end{rem}
\medskip

\noindent{\bf The blow-up of the polynomial space at the boundary point.}
We turn to the polyanalytic analogue of the left-hand side of
Lemma \ref{analyticbdrblowup}. 

\begin{defn} We introduce the {\em blow-up space at} $1$,
\[
\mathrm{Pol}^{\langle1\rangle}_{m,n,q} := 
\big\{\e^{-{m}^{1/2}\xi}p(1+m^{-1/2}\xi):\,\, p \in \mathrm{Pol}_{m,n,q} 
\big\},
\]
which we equip with the norm 
\begin{equation*} 
\big\|\xi\mapsto \e^{-{m}^{1/2}\xi}p(1+m^{-1/2}\xi)
\big\|_{{\Fspace}^{\langle1\rangle}_{m,n,q}}:=m^{1/2}\e^{\frac12m}
\|p\|_{L^2(\C,\e^{-m|z|^2})}.
\end{equation*}
For $0\le r\le q-1$, we denote by $\Fspace^{\langle1\rangle}_{m,n,r+1}$ the 
subspace
\[
\Fspace^{\langle1\rangle}_{m,n,r+1} := 
\big\{\e^{-{m}^{1/2}\xi}p(1+m^{-1/2}\xi):\,\, p \in \Fspace_{m,n,r+1} 
\big\},
\]
equipped with the same norm.
\end{defn}

An elementary change of variables argument allows us to identify the norm on
$\mathrm{Pol}^{\langle1\rangle}_{m,n,q}$ with that of $A^2_{m,q}(\C)$:
\begin{equation} 
\label{normidentity}
\big\|\xi\mapsto \e^{-{m}^{1/2}\xi}p(1+m^{-1/2}\xi)
\big\|_{\mathrm{Pol}^{\langle1\rangle}_{m,n,r+1}}=
\big\|\xi\mapsto \e^{-{m}^{1/2}\xi}p(1+m^{-1/2}\xi)
\big\|_{A^2_{1,r+1}(\C)}.
\end{equation}
As a consequence, we may regard $\mathrm{Pol}^{\langle1\rangle}_{m,n,q}$ 
and ${\Fspace}^{\langle1\rangle}_{m,n,r+1}$ as norm closed subspaces of 
$A^2_{1,q}(\C)$. We may read off from the definition 
of the norm in ${\Fspace}^{\langle1\rangle}_{m,n,r+1}$ that the kernel on the 
left-hand side in Lemma \ref{analyticbdrblowup} is the reproducing kernel for 
the space $\mathrm{Pol}^{\langle1\rangle}_{m,n,1}$. So, 
Lemma \ref{analyticbdrblowup} can be understood as saying that
\begin{equation} 
\label{analytictildeconvergence}
K_{\mathrm{Pol}^{\langle1\rangle}_{m,n,1}}(\xi,\eta)= 
K_{\Bop_0[L^2(\R_-)]}(\xi,\eta)+\Ordo(m^{-\frac12+\varepsilon}), 
\end{equation}
where $K_{\mathrm{Pol}_{m,n,1}}$ and $K_{\Bop_0[L^2(\R_-)]}$ denote the 
reproducing kernels of the spaces in the subscripts, and the bound is locally 
uniform on compact subsets. We want to generalize 
\eqref{analytictildeconvergence} beyond $r=0$. To this end, we make use of
the operators $\Tope_{r}$.

\begin{prop}
\label{lem-T1}
For $r=1,2,3,\ldots$, we have that
\[
r^{-1/2}\Tope_{1}:\,{\Fspace}^{\langle1\rangle}_{m,n,r}\to
{\Fspace}^{\langle1\rangle}_{m,n,r+1}
\]
is an isometric isomorphism.
\end{prop}

\begin{proof}
From the isometry properties of $\Tope_{r-1}$ and $\Tope_{r}$ together 
with the semi-group property 
$r^{-1/2}\Tope_{1}\circ\Tope_{r-1}=\Tope_{r}$, we get that 
\[
r^{-1/2}\Tope_{1}:\,\delta A^2_{1,r}(\C)\to\delta A^2_{1,r+1}(\C)
\]
is an isometric isomorphism. In view of \eqref{normidentity}, the isometry
part of the assertion follows. It remains to show that the operator is onto.
This is an algebraic exercise which we leave to the reader. 
%
\end{proof}

By iterating Proposition \ref{lem-T1}, we obtain the following.

\begin{cor} 
\label{tildecor}
For $r=1,2,3\ldots$, we have that
\[ 
\Tope_r:\mathrm{Pol}^{\langle1\rangle}_{m,n,1}\to
\Fspace^{\langle1\rangle}_{m,n,r+1} 
\]
is an isometric isomorphism.
\end{cor}

\medskip

\noindent{\bf The blow-up of the polynomial reproducing kernel at a boundary
point.}
From Corollary \ref{tildecor} above, we get that
\begin{equation}
\label{eq-kerneltop}
K_{\Fspace^{\langle1\rangle}_{m,n,r+1}}(\xi,\eta)=[\Tope_r]_\xi
[\bar\Tope_r]_\eta
\big(K_{\mathrm{Pol}^{\langle1\rangle}_{m,n,1}}(\xi,\eta)\big),
\end{equation}
where the subscripts $z$ and $w$ are used to indicate that the operator is
acting with respect to that variable, and the bar means complex
conjugation of the operator. To be more precise,
\[
\bar\Tope_r[f](z):=\frac{1}{\sqrt{r!}} \e^{|z|^2}\,
\bar\partial_z^r\big\{f(z)\e^{-|z|^2}\big\}.
\]
We would like to plug in the approximation \eqref{analytictildeconvergence}
into \eqref{eq-kerneltop}. The operator $\Tope_r$ is a sum of certain 
polynomials in $\bar z$ of degree $\le r$ times the differential operator 
$\partial_z$ to powers $\le r$.  
The Cauchy integral formula allows us to control the size of the derivatives 
on a compact subset in terms of the size of the functions on a slightly bigger
compact subset. This means that the approximation 
\eqref{analytictildeconvergence} carries over, and we find that
\begin{equation} 
\label{analytictildeconvergence2}
K_{\Fspace^{\langle1\rangle}_{m,n,r+1}}(\xi,\eta)= 
[\Tope_r]_\xi
[\bar\Tope_r]_\eta
\big(K_{\Bop_0[L^2(\R_-)]}(\xi,\eta)\big)+\Ordo(m^{-\frac12+\varepsilon}), 
\end{equation}
with uniform control on compact subsets. Next, as 
\[
\mathrm{Pol}^{\langle1\rangle}_{m,n,q}=\bigoplus_{r=0}^{q-1}
\Fspace^{\langle1\rangle}_{m,n,r+1},
\]
we get that
\begin{multline}
\label{analytictildeconvergence3}
K_{\mathrm{Pol}^{\langle1\rangle}_{m,n,q}}(\xi,\eta)=
\sum_{r=0}^{q-1}K_{\Fspace^{\langle1\rangle}_{m,n,r+1}}(\xi,\eta)
\\
=\sum_{r=0}^{q-1}[\Tope_r]_\xi[\bar\Tope_r]_\eta
\big(K_{\Bop_0[L^2(\R_-)]}(\xi,\eta)\big)+\Ordo(m^{-\frac12+\varepsilon}), 
\end{multline}
again with uniform control on compact subsets. Next, it should be rather clear
that
\[
[\Tope_r]_\xi[\bar\Tope_r]_\eta\big(K_{\Bop_0[L^2(\R_-)]}(\xi,\eta)\big)
\]
is the reproducing kernel for the space $\Tope_r\Bop_0[L^2(\R_-)]=\Bop_r
[L^2(\R_-)]$, which was identified in terms of Hermite polynomials back in
Proposition \ref{polybargmannkernel}.
We write this down as a proposition. 

\begin{prop}
\label{prop-5.9}
Fix a positive real number  $\varepsilon$. Then the reproducing kernel for 
$\mathrm{Pol}^{\langle1\rangle}_{m,n,q}$ has the following form:
\[
K_{\mathrm{Pol}^{\langle1\rangle}_{m,n,q}}(\xi,\eta)=
\sum_{r=0}^{q-1}\frac{\e^{\xi\bar\eta}}{r! \sqrt{2 \pi}}
\int_{-\infty}^{-\xi-\bar\eta} H_r(t+\xi-\eta)
H_r(t+\bar\eta-\bar\xi)\e^{-\frac12t^2} \diff t+
\Ordo(m^{-\frac12+\varepsilon}),
\]
as $m,n\to+\infty$ while $n=m+\Ordo(1)$, where the control is uniform on 
compact subsets. 
\end{prop}

If we like, we may use the classical Christoffel-Darboux identity
\begin{equation}
\label{eq-CD}
\sum_{r=0}^{q-1}\frac{1}{r!}H_r(x)H_r(y)=
\frac{H_q(x)H_{q-1}(y)-H_{q-1}(x)H_{q}(y)}{(q-1)!(x-y)}
\end{equation}
to rewrite the above sum. Also, we should note that reproducing kernel for
the blow-up space $\mathrm{Pol}^{\langle1\rangle}_{m,n,q}$ is connected with
the reproducing kernel $K_{m,n,q}$ for $\mathrm{Pol}_{m,n,q}$ via
the identity
\begin{equation}
\label{eq-connect}
K_{\mathrm{Pol}^{\langle1\rangle}_{m,n,q}}(\xi,\eta)=
m^{-1}\e^{-m-m^{1/2}(\xi+\bar\eta)}K_{m,n,q}(1+m^{-1/2}\xi,1+m^{-1/2}\eta).
\end{equation}
\medskip

\noindent{\bf The blow-up of the $1$-point intensity near a boundary point.}
The $1$-point intensity function is
\[
K_{m,n,q}(z,z)\e^{-m|z|^2},
\]
and the localized version with $z=1+m^{-1/2}\xi$ is
\[
U_{m,n,q}(\xi):=
m^{-1}K_{m,n,q}(1+m^{-1/2}\xi,1+m^{-1/2}\xi)\,\e^{-m|1+m^{-1/2}\xi|^2},
\]
where we throw in a factor of $m^{-1}$ to compensate for the Jacobian.
In view of \eqref{eq-connect} together with Proposition \ref{prop-5.9},
we obtain
\[
U_{m,n,q}(\xi)=\sum_{r=0}^{q-1}\frac{1}{r!\sqrt{2\pi}}\int_{-\infty}^{-2\re\xi}
H_r(t)^2\e^{-\frac12t^2}\diff t+\Ordo(m^{-\frac12+\varepsilon}).
\]
So, essentially, the $1$-point intensity function is determined by the density 
\[
t\mapsto\sum_{r=0}^{q-1}\frac{1}{r!\sqrt{2\pi}}
H_r(t)^2\e^{-\frac12t^2},
\]
which corresponds to filling the lowest energy eigenstates of the harmonic
oscillator. By the Wigner semi-circle law, then, we get the approximation
\[
U_{m,n,q}(\xi)\approx
\frac{2q}{\pi}\int_{-1}^{-q^{-1/2}\re\xi}\sqrt{1-\tau^2}\diff\tau,
\]
valid for big $m,n$ with $n=m+\Ordo(1)$, and big $q$ (but much smaller than
$m,n$). So, if we rescale to characteristic distance $q^{1/2}m^{-1/2}$
we find an interesting law in the limit. 

\medskip

\noindent{\bf The blow-up Berezin density at a boundary point.}
The blow-up Berezin density at $1$ is given by
\[
\hat\berd^{\langle1\rangle}_{m,n,q}(\xi)=m^{-1}
\berd^{\langle1\rangle}_{m,n,q}(1+m^{-1/2}\xi)=m^{-1}
\e^{-m|1+m^{-1/2}\xi|^2}
\frac{|K_{m,n,q}(1+m^{-1/2}\xi,1)|^2}{K_{m,n,q}(1,1)}.
\] 
From \eqref{eq-connect}, we have that
\begin{equation*}
K_{\mathrm{Pol}^{\langle1\rangle}_{m,n,q}}(0,0)=
m^{-1}\e^{-m}K_{m,n,q}(1,1),
\end{equation*}
while Proposition \ref{prop-5.9} gives 
\[
K_{\mathrm{Pol}^{\langle1\rangle}_{m,n,q}}(0,0)=
\sum_{r=0}^{q-1}\frac{1}{r! \sqrt{2 \pi}}
\int_{-\infty}^{0} H_r(t)^2\e^{-\frac12t^2} \diff t+
\Ordo(m^{-\frac12+\varepsilon}),
\]
as $m,n\to+\infty$ with $n=m+\Ordo(1)$.
Now, as each Hermite polynomial $H_r$ is either even or odd, 
\[
\int_{-\infty}^{0} H_r(t)^2\e^{-\frac12t^2} \diff t=\frac{1}{2}
\int_{-\infty}^{+\infty} H_r(t)^2\e^{-\frac12t^2} \diff t=
\frac{r!\sqrt{2\pi}}{2},
\]
which leads to 
\[
K_{\mathrm{Pol}^{\langle1\rangle}_{m,n,q}}(0,0)=
\frac{q}{2}+\Ordo(m^{-\frac12+\varepsilon})
\]
and  
\begin{equation*}
K_{m,n,q}(1,1)=\tfrac12\,mq\e^m[1+\Ordo(m^{-\frac12+\varepsilon})].
\end{equation*}
A similar calculation gives that
\begin{multline*}
K_{m,n,q}(1+m^{-1/2}\xi,1)=m\e^{m+m^{1/2}\xi}
K_{\mathrm{Pol}^{\langle1\rangle}_{m,n,q}}(\xi,0)
\\
=m\e^{m+m^{1/2}\xi}\bigg\{
\sum_{r=0}^{q-1}\frac{1}{r! \sqrt{2 \pi}}
\int_{-\infty}^{-\xi} H_r(t+\xi)
H_r(t-\bar\xi)\e^{-\frac12t^2} \diff t+
\Ordo(m^{-\frac12+\varepsilon})\bigg\}.
\end{multline*}

Putting things together, we obtain the following asymptotics for the blow-up 
Berezin density.

\begin{thm}
\label{thm-5.10}
Fix a positive real number $\varepsilon$. Then the blow-up Berezin density at
$1$ has the following form:
\[
\hat\berd^{\langle1\rangle}_{m,n,q}(\xi)=
\frac{1}{\pi q}\bigg|\sum_{r=0}^{q-1}\frac{1}{r!}
\int_{-\infty}^{-\xi} H_r(t+\xi)
H_r(t-\bar\xi)\e^{-\frac12t^2} \diff t\bigg|^2+
\Ordo(m^{-\frac12+\varepsilon}),
\]
as $m,n\to+\infty$ while $n=m+\Ordo(1)$, where the control is uniform on 
compact subsets. 
\end{thm}

\begin{rem}
When we make some explicit calculations based on Theorem \ref{thm-5.10}, 
we see that the Fresnel zone pattern is less pronounced for a boundary point
$z$.
\end{rem}
\medskip

\noindent\bf Acknowledgements. \rm 
We would like to thank Philip Kennerberg for sharing sharing his 
MATLAB-code and discussing the practicalities in simulating 
the process. 


\end{document}